\documentclass[12pt]{amsart}

\usepackage{amssymb}

\usepackage{graphicx}

\numberwithin{equation}{section}

\newtheorem{theorem}{Theorem}[section]
\newtheorem{corollary}[theorem]{Corollary}
\newtheorem{lemma}[theorem]{Lemma}
\newtheorem{proposition}[theorem]{Proposition}

\theoremstyle{definition}
\newtheorem{definition}[theorem]{Definition}
\newtheorem{remark}[theorem]{Remark}

\def\R{\mathbb R}

\newcommand{\ep}{\varepsilon}

\title[Regularity of interfaces via obstacle problems]{Regularity of interfaces in phase transitions via obstacle problems}

\author[Alessio Figalli]
{Alessio Figalli}
\address{ETH Z\"urich, Mathematics Department, R\"amistrasse 101, 8092 Z\"urich, Switzerland.}
\email{alessio.figalli@math.ethz.ch}

\begin{document}

\begin{abstract}
The aim of this note is to review some recent developments on the regularity theory for the stationary and parabolic obstacle problems. 

After a general overview, we present some recent results on the structure of singular free boundary points. Then, we show some selected applications to the generic smoothness of the free boundary in the stationary obstacle problem (Schaeffer's conjecture), and to the  smoothness of the free boundary in the one-phase Stefan problem for almost every time. 
\end{abstract}

\maketitle

\tableofcontents

\section{Introduction}

\subsection{The classical Stefan problem}
The classical Stefan problem aims to describe the temperature distribution in a homogeneous medium undergoing a phase change, typically the melting of a body of ice maintained at zero degrees centigrade. Given are the initial temperature distribution of the water and the energy contributed to the system through the boundary of the domain.
The unknowns are the temperature distribution of the water as a function of space and time, and the  ice-water interface.

This problem is named after Josef Stefan, a Slovenian physicist who introduced the general class of such problems around 1890 in relation to problems of ice formation, although this question had already been considered by Lam\'e and Clapeyron in 1831.

\smallskip

In its most classical formulation, the Stefan 
problem can be formulated as follows:
let $\Omega\subset \R^n$ be a bounded domain, and let $$\theta=\theta(t,x)$$ denote the {temperature} of the medium at a point $x \in \Omega$ at time $t \in \R^+$.
We assume that $\theta\geq 0$ in $\R^+\times \Omega$, so that 
 $\{\theta=0\}$ represents the ice while $\{\theta>0\}$ represents the water,
see Figure \ref{Pic01}.
 \begin{figure}[ht]
 \includegraphics[scale=0.25]{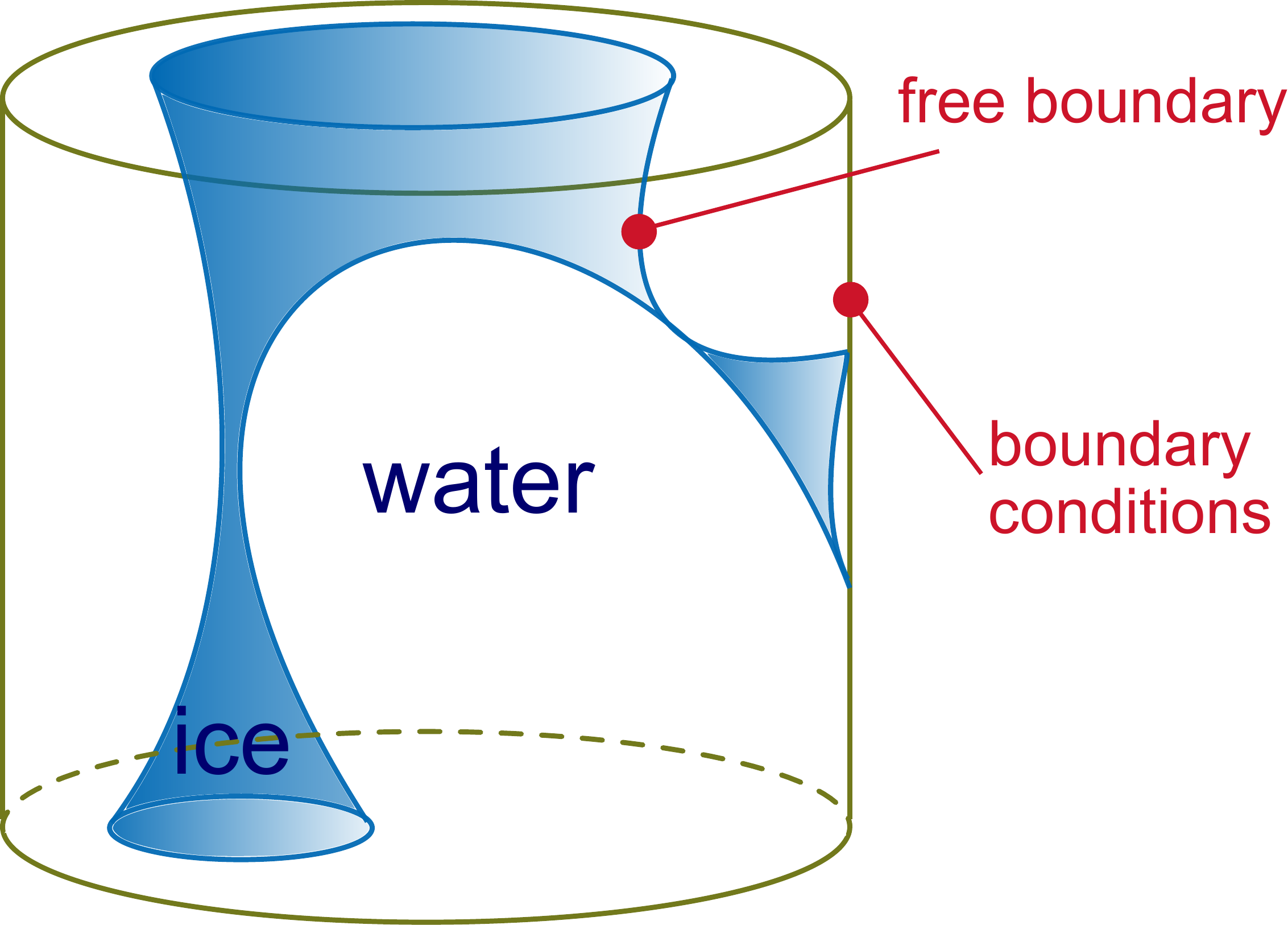}
 \caption{The Stefan problem. Because the ice-water interface is an unknown of the problem, it is called ``free boundary''.} 
  \label{Pic01}
 \end{figure}
 
We prescribe an initial condition
$$\theta(0,x)=\theta_0(x)\geq 0$$
at time $t=0$, and a boundary condition
$$
\theta(t,x)=\theta_b(t,x)\geq 0 \qquad \text{for $x \in \partial\Omega$ and $t\geq 0$}.
$$
In the water, the temperature evolves in time according to the classical heat equation, that is
\begin{equation}
\label{eq:heat}
\partial_t\theta=\Delta \theta\qquad \text{inside }\{\theta>0\}.
\end{equation}
Also, the interface ice-water moves accordingly to the so-called ``Stefan condition''
\begin{equation}
\label{eq:stefan}
\dot x(t)=-\nabla \theta(t,x(t))\qquad \forall\, x(t)\in \partial \{\theta(t)>0\},
\end{equation}
where $\nabla\theta(t)$ denotes the spatial gradient of $\theta(t)$ computed from inside the region $\{\theta(t)>0\}$, see Figure \ref{Pic02}.
\bigskip

\begin{figure}[ht]
 \includegraphics[scale=0.4]{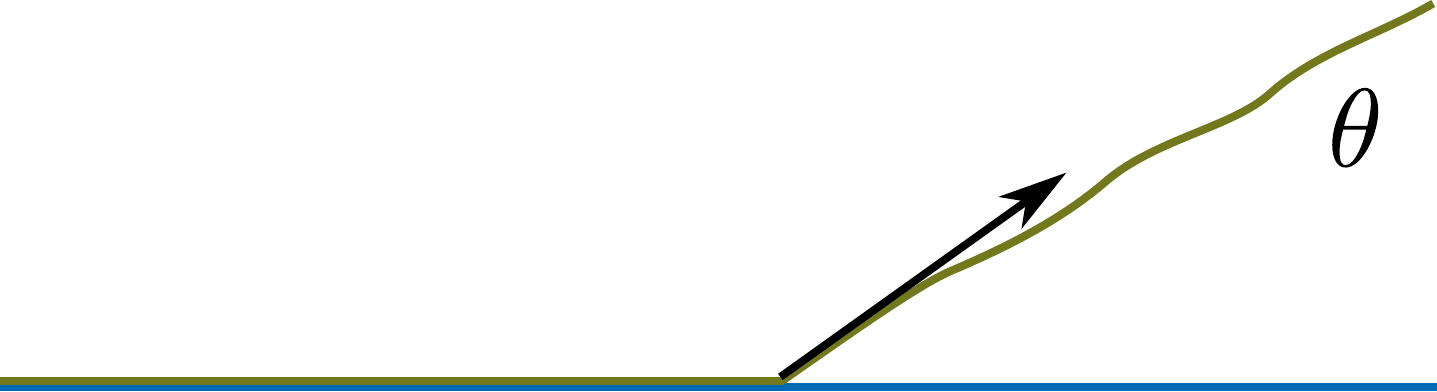}
 \caption{The gradient of $\theta(t)$ at a free boundary point is computed from inside the region $\{\theta(t)>0\}$.} 
 \label{Pic02}
 \end{figure}
Observe that $-\nabla \theta(t,x(t))$ points always towards the region $\{\theta(t)=0\}$, hence this set shrinks in time, see Figure \ref{Pic03}.  In other words, ice is melting.
 \begin{figure}[ht]
 \includegraphics[scale=0.14]{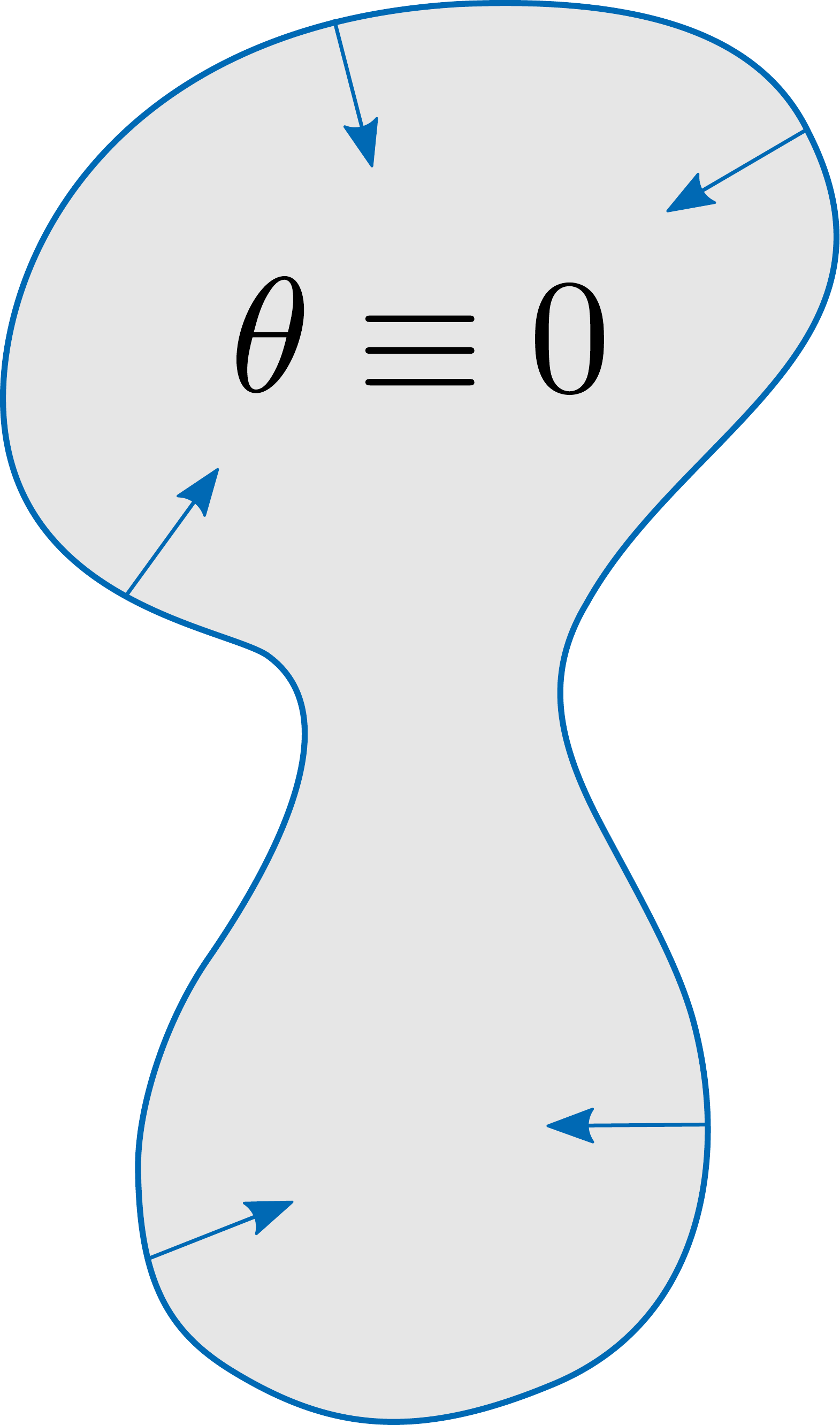}
 \caption{
 The vector $-\nabla \theta$ points into $\{\theta = 0\}$, which shrinks as time 
evolves.
} 
  \label{Pic03}
 \end{figure}
 
This problem belongs to the general class of ``free-boundary problems'', since $\theta$ solves a PDE (the heat equation) inside the time-evolving domain $\{\theta>0\}$ which depends on the solution itself, and so in particular it is an unknown of the problem. In this regard, we say that the boundary $\partial\{\theta>0\}$ is a {\it free-boundary}.

\subsection{From Stefan to the parabolic obstacle problem}
\label{sec:stefan obst}
In order to study this problem it is convenient to perform the so-called Duvaut's transformation \cite{Duv,Duv2}: let
$$
u(t,x):=\int_0^t\theta(s,x)\,ds.
$$
Then $u: \R^+\times \Omega\to \R$ solves the so-called  {\it parabolic obstacle problem}
 \begin{equation}
 \label{eq:parab obst}
\partial_tu=\Delta u-\chi_{\{u>0\}},\qquad u \geq 0,\qquad \partial_tu \geq 0,
\end{equation}
with Dirichlet boundary conditions
$$
u(t,x)=\int_0^t\theta_b(s,x)\,ds\qquad \text{ for  $x\in \partial\Omega$ and $t\geq 0$}
$$
(we shall explain the name ``obstacle problem'' in the next section).

To understand how to obtain \eqref{eq:parab obst} from the Stefan problem, we give here an informal derivation assuming that the set $\partial\{\theta>0\}$  can be represented as the graph of a smooth function $\tau:\Omega\to \R^+$, that is
$$
\partial\{\theta>0\}=\{(t,x)\,:\,t=\tau(x)\}={\rm graph}(\tau).
$$
In other words, $\tau(x)$ represents the moment when the ice present at $x$ melts into water.

Since
\begin{equation}
\label{eq:theta tau}
\theta(\tau(x),x)=0\qquad \forall\, x \in \Omega,
\end{equation}
differentiating this relation with respect to $x$ we obtain
\begin{equation}
\label{eq:ste ob}
0=\nabla\bigl[\theta(\tau(x),x)\bigr]=\partial_t\theta(\tau(x),x)\, \nabla\tau(x)+\nabla\theta(\tau(x),x).
\end{equation}
Also, since
$$
\theta(t,x(t))=0 \quad \text{ for any curve $t\mapsto x(t)\in \partial\{\theta(t)>0\},$}
$$
differentiating this relation in time and using the Stefan condition \eqref{eq:stefan}, we get
\begin{equation}
\label{eq:ste ob2}
0=\frac{d}{dt}\theta(t,x(t))
=\partial_t\theta+\nabla\theta\cdot \dot x(t)=\partial_t\theta-|\nabla\theta|^2\qquad \text{on } \partial\{\theta>0\}.
\end{equation}
Hence, combining \eqref{eq:ste ob} and \eqref{eq:ste ob2}, we deduce that
\begin{equation}
\label{eq:ste ob3}
\nabla \tau(x)\cdot \nabla\theta(\tau(x),x)
=-\frac{|\nabla\theta(\tau(x),x)|^2}{\partial_t \theta(\tau(x),x)}=-1.
\end{equation}
Note that, because the sets $\{\theta(t)=0\}$ shrink in time,  we have
$$
\theta(t,x)=0\qquad \text{for $t \in [0,\tau(x)]$.}
$$
In particular 
$$
u(t,x)=\int_0^t\theta(s,x)\,dx=0 \qquad \text{for $t \leq \tau(x)$,}
$$
which implies that
\begin{equation}
\label{eq:u}
u(t,x)=\int_{\tau(x)}^t\theta(s,x)\,ds\qquad \forall\,t>\tau(x)
\end{equation}
and that
\begin{equation}
\label{eq:u 0}
\{\theta>0\}=\{(t,x)\,:\,t>\tau(x)\}= \{u>0\}.
\end{equation}
We now want to compute the equation for $u$ for $t >\tau(x)$.

Differentiating \eqref{eq:u} with respect to $x_i$ and recalling \eqref{eq:theta tau}, we obtain
$$
\partial_{x_i}u(t,x)=
\int_{\tau(x)}^t\partial_{x_i}\theta(s,x)\,ds-\theta(\tau(x),x)\,\partial_{x_i}\tau(x)
=\int_{\tau(x)}^t\partial_{x_i}\theta(s,x)\,ds.
$$
Differentiating again with respect to $x_i$ yields
$$
\partial_{x_ix_i}u(t,x)=
\int_{\tau(x)}^t\partial_{x_ix_i}\theta(s,x)\,ds-\partial_{x_i}\theta(\tau(x),x)\,\partial_{x_i}\tau(x),
$$
so that summing over $i=1,\ldots,n$ gives
$$
\Delta u(t,x)=
\int_{\tau(x)}^t\Delta \theta(s,x)\,ds-\nabla\theta(\tau(x),x)\cdot\nabla\tau(x).
$$
Hence, since $\theta=\partial_tu$, recalling \eqref{eq:heat}, \eqref{eq:theta tau}, and \eqref{eq:ste ob3}, we obtain
$$
\Delta u(t,x)=
\int_{\tau(x)}^t\partial_t\theta(s,x)\,ds+1=\theta(t,x)+1=
\partial_tu(t,x)+1
$$
inside the region $\{\theta>0\}$.
Recalling \eqref{eq:u 0},
 we proved that
\begin{equation}
\label{eq:u in u pos}
\partial_tu=\Delta u-1\qquad \text{inside }\{u>0\}.
\end{equation}

\smallskip

It is important to remark that \eqref{eq:u in u pos} is not equivalent to \eqref{eq:parab obst}, as the latter equation has a further hidden condition on the free boundary.
To understand this, assume that $v$ solves \eqref{eq:parab obst}. Then, since 
 $$\partial_t v-\Delta v=-\chi_{\{v>0\}}\in L^\infty,$$ it follows by parabolic regularity theory that $v(t) \in C^1(\Omega)$ for any $t>0$. In particular, since $v(t)\geq 0$ and $v(t)|_{\partial\{v(t)>0\}}=0$,  it holds
$$
 \nabla v(t)= 0\qquad \text{on }\partial\{v(t)>0\}.
$$
Hence, we need to show that the function $u(t,x)$ defined in \eqref{eq:u} satisfies also this extra condition. 

To prove this, we recall (see \eqref{eq:theta tau}) that 
$$
0=\theta(\tau(x),x)=\partial_tu(\tau(x),x).
$$
Also, since $u(\tau(x),x)=0$, differentiating this relation with respect to $x$ we get
$$
0=\nabla\bigl[u(\tau(x),x)\bigr]
=\partial_tu(\tau(x),x)\,\nabla\tau(x)+\nabla u(\tau(x),x).
$$
Combining these two relations we obtain
$$
\nabla u(\tau(x),x)= 0,
$$
that is
 \begin{equation}
 \label{eq:zero grad t}
 \nabla u(t)= 0\qquad \text{on }\partial\{u(t)>0\},
 \end{equation}
see Figure \ref{Pic-par-ob}.
  \begin{figure}[ht]
\includegraphics[scale=0.25]{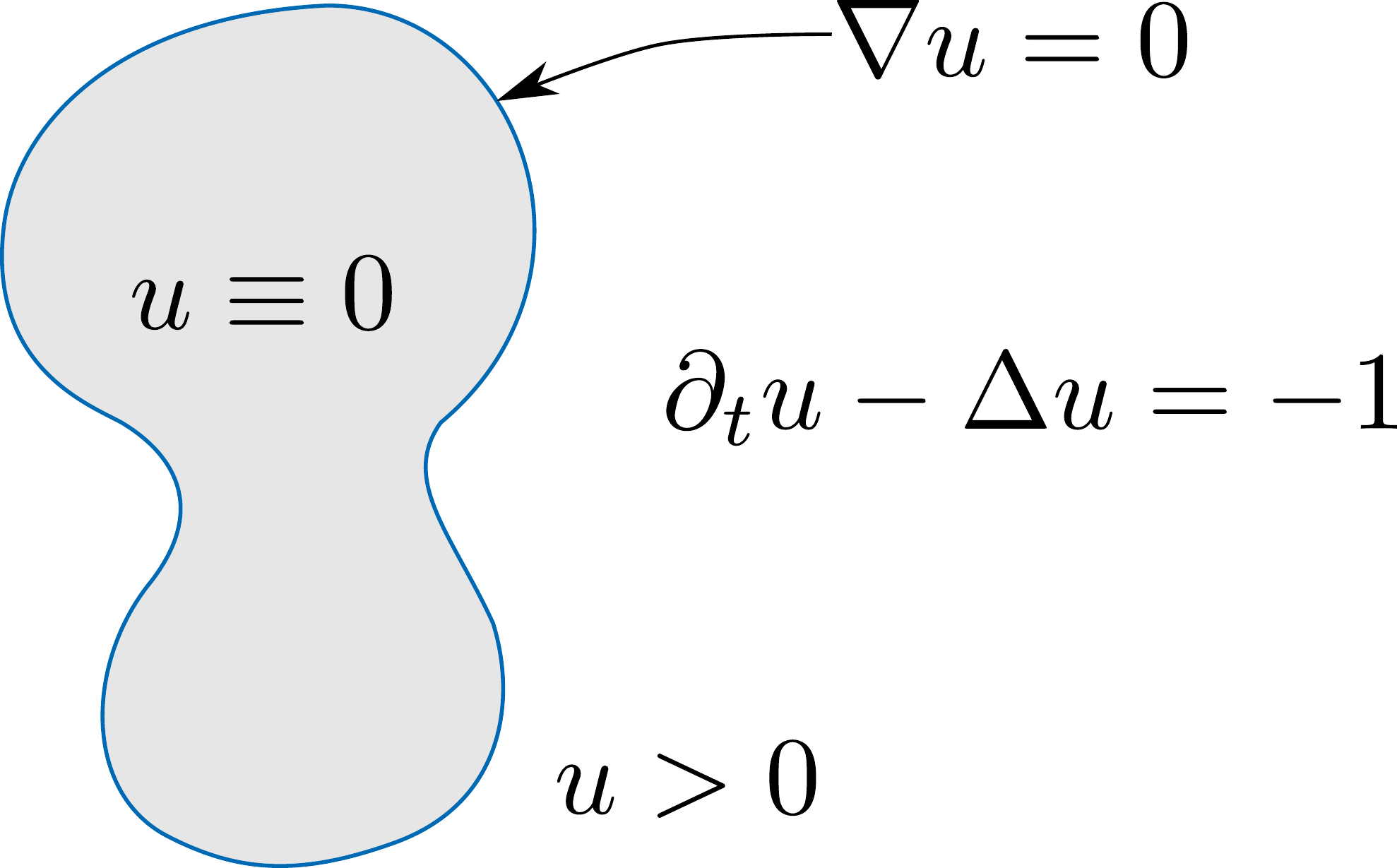} \caption{The parabolic obstacle problem.} 
  \label{Pic-par-ob}
 \end{figure}

\smallskip

It is interesting to observe that, in the equation for $u$, no transmission condition appears on the free boundary (cp. the Stefan condition for $\theta$, see \eqref{eq:stefan}).
Hence, one may wonder what is determining the evolution of the free boundary. This can be explained as follows: 
 recalling \eqref{eq:zero grad t}, we showed that both $u(t)$ and $\nabla u(t)$ vanish on the free boundary for any $t > 0$. One may notice that this is a very strong condition, since we are saying that $u(t)$ solves the parabolic equation
 $$
 \partial_tu=\Delta u-1 \qquad\text{in }\{u>0\}
 $$
 with {\em two} boundary conditions on the free boundary:
 $$
 u=0 \quad \text{and}\quad \nabla u=0\qquad \text{on }\partial\{u>0\},
 $$
 see Figure \ref{Pic-par-ob}.
This is an over-determined system: in classical PDE theory one can only prescribe either Dirichlet or Neumann boundary conditions, but not both! This means that the free boundary has to evolve so to ensure that both boundary conditions hold for every time. In other words, the transmission condition determining the evolution of the free boundary is now hidden inside the equation \eqref{eq:parab obst} via the condition \eqref{eq:zero grad t}.

\smallskip

Thanks to this informal discussion, since $\{\theta>0\}=\{u>0\}$, we have reduced the study of the free boundary in the Stefan problem to the one in the parabolic obstacle problem. Our goal
now is to understand the free boundary regularity in the parabolic obstacle problem.
In order to simplify the analysis, it makes sense to start first from the stationary case where $u$ is independent of time, and then move to the general case. 
This will be the focus of the next sections.

\section{The elliptic obstacle problem}
In this section we study the free boundary regularity in the stationary obstacle problem.

Thus, given a domain $\Omega\subset \R^n$ and some fixed smooth boundary conditions $f:\partial\Omega\to \R$ with $f>0$, we want to investigate the {\em elliptic obstacle problem}
\begin{equation}
\label{eq:obst}
\left\{
\begin{array}{ll}\Delta u=\chi_{\{u>0\}}&\text{in }\Omega,\\
u \geq 0&\text{in }\Omega,\\
u=f &\text{on }\partial\Omega.
\end{array}
\right.
\end{equation}
Note that, because $\Delta u=\chi_{\{u>0\}} \in L^\infty(\Omega)$, it follows by elliptic regularity that $u \in C^1(\Omega)$. In particular, as in the previous section, since $u\geq 0$ and $u|_{\partial\{u>0\}}=0$, one deduces that 
 \begin{equation}
 \label{eq:zero grad}
 \nabla u= 0\qquad \text{on }\partial\{u>0\},
 \end{equation}
see Figure \ref{Pic-ell-ob}.

 \begin{figure}[ht]
\includegraphics[scale=0.25]{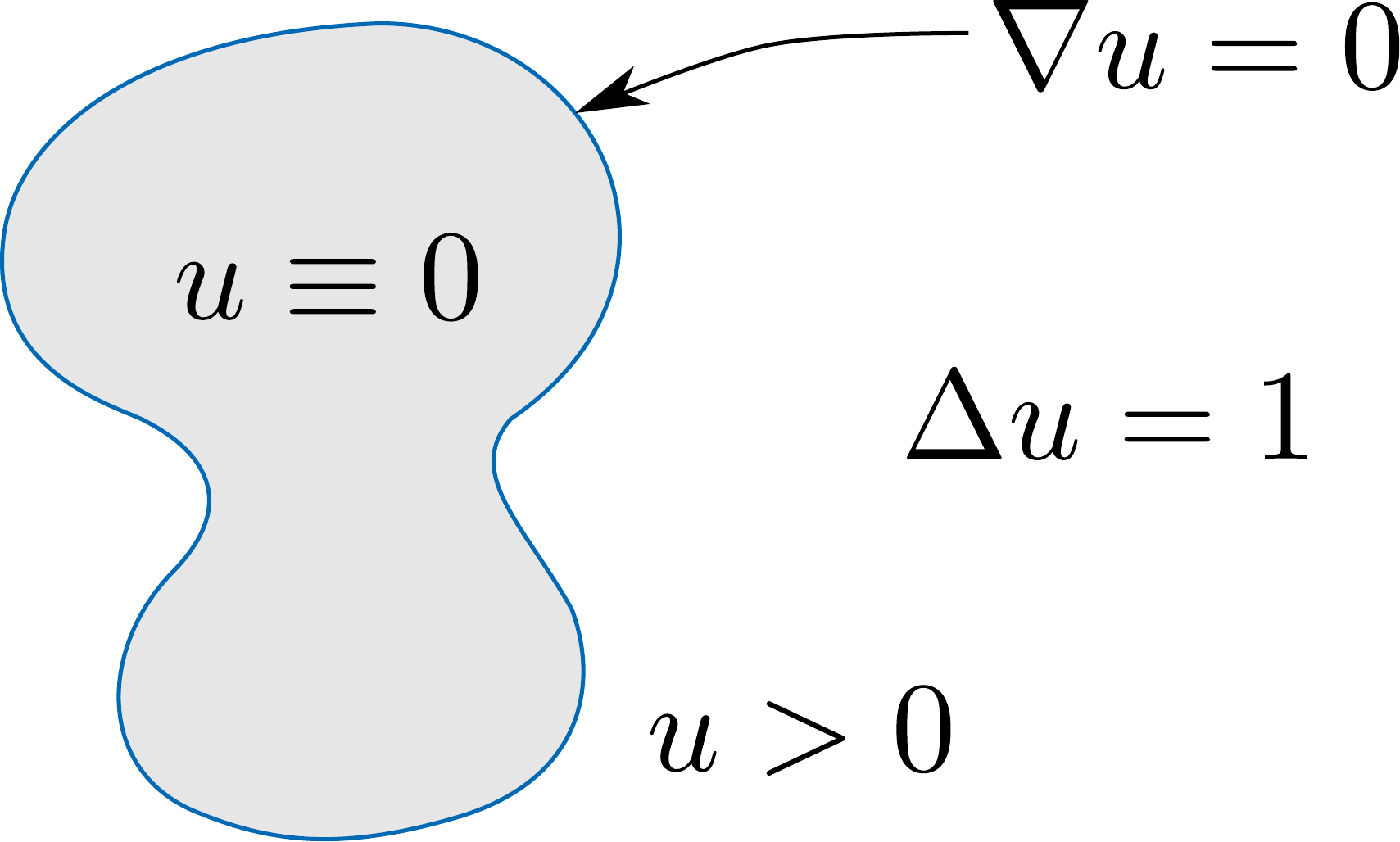} \caption{The stationary obstacle problem.} 
  \label{Pic-ell-ob}
 \end{figure}

Before beginning our study, we first want to explain the origin of the name ``obstacle problem'' associated to \eqref{eq:obst}.

Consider an elastic membrane that coincides with the graph of $f$ on $\partial\Omega$, subject to the action of gravity, and forced to lie above the plane
 $\{x_{n+1}= 0\}$ (the ``obstacle'').
If we represents the membrane as the graph of a nonnegative function $u:\Omega\to \R$, this function minimizes the functional 
\begin{equation}
\label{eq:min obst}
\min_{v\geq 0}\biggl\{\int_{\Omega} \biggl(\frac{|\nabla v|^2}{2}+g\,v\biggr) \,:\,v|_{\partial\Omega}=f \biggr\},
\end{equation}
where\\
$$
\int_{\Omega} \frac{|\nabla v|^2}{2}\quad \text{ represents the elastic energy of the graph of $v$,}
$$
and
$$g\int_{\Omega} v \quad \text{ represents the gravitational energy of the graph of $v$}
$$
 (here $g>0$ is the gravitational constant), see Figure \ref{Pic04}.
\bigskip

 \begin{figure}[ht]
\includegraphics[scale=0.37]{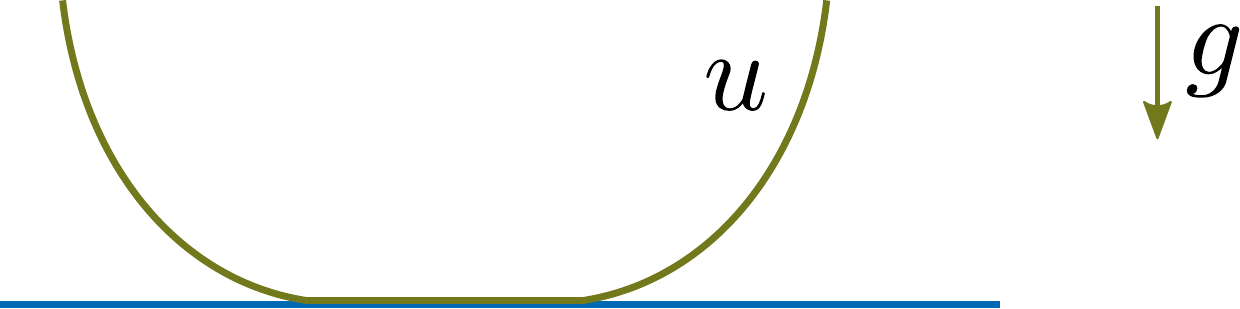}
\caption{The obstacle problem models an elastic surface lying above the plane $\{x_{n+1}= 0\}$ and subject to gravity.} 
\label{Pic04}
\end{figure}

The existence and uniqueness of a minimizer for \eqref{eq:min obst} follows by standard techniques in the calculus of variations, see for instance \cite[Section 2]{F-JEDP}.
Then, computing the Euler-Lagrange equations for the minimizer $u$, one can prove that $u$ satisfies the equation
$$
\Delta u=g\,\chi_{\{u>0\}},
$$
see for instance \cite[Section 3]{F-JEDP}.
In particular, up to replacing $u$ by $u/g$ we can assume that $g=1$, which shows that \eqref{eq:obst} corresponds to the Euler-Lagrange equations associated to the minimization problem \eqref{eq:min obst}.
This justifies the name {\it obstacle problem}.
Also, since the set $\{u=0\}$ corresponds to the region where $u$ touches the obstacle, one refers to $\{u=0\}$ as the {\it contact set}.

\subsection{Regularity properties of $u$}
Let $u$ solve \eqref{eq:obst}.
Since $$\Delta u=\chi_{\{u>0\}} \in L^\infty(\Omega),$$ standard elliptic regularity (see for instance \cite[Corollary 9.10 and Theorem 9.13]{GT01}) guarantees that $u \in W_{\rm loc}^{2,p}(\Omega)$ for any $p<\infty$. In other words
$$
D^2u \in L^p_{\rm loc}(\Omega)\qquad \forall\,p<\infty.
$$
A key result in the theory of obstacle problems states that the estimate above holds even for $p=\infty$, see \cite{Fre72,BK74,C98,F-JEDP}:
\begin{theorem}
\label{thm:C11}
Let $u$ solve \eqref{eq:obst}.
Then
$$
D^2u \in L^\infty_{\rm loc}(\Omega).
$$
\end{theorem}
It is worth noticing that  the  result above is optimal.
Indeed, since $\Delta u=\chi_{\{u>0\}}$ is discontinuous, the statement $u \in C^2(\Omega)$ is clearly false. Thus the boundedness of $D^2u$ is the best one can hope for.

As a consequence of 
Theorem \ref{thm:C11}, we deduce that $u$ grows at most quadratically away from the free boundary.
\begin{corollary}
\label{cor:C11}
Let $u$ solve \eqref{eq:obst},
let $\Omega'\subset \subset \Omega$,
let $x_0 \in \partial\{u>0\}$, and assume that $B_{r}(x_0)\subset \Omega'$.
Then there exists $C=C(\Omega')$ such that
$$
0\leq \sup_{B_r(x_0)}u\leq C\,r^2.
$$
\end{corollary}
\begin{proof}
Given $x \in B_r(x_0)$, we can write $u(x)$ using the Taylor formula centered at $x_0$:
\begin{multline*}
u(x)=u(x_0)+\nabla u(x_0)\cdot (x-x_0)\\
+\int_0^1(1-t)D^2u\bigl(x_0+t(x-x_0)\bigr)[x-x_0,x-x_0]\,dt.
\end{multline*}
Since $u(x_0)=0$ and $\nabla u(x_0)=0$,
setting  $C:=\|D^2u\|_{L^\infty(\Omega')}$ 
we get
$$
0\leq u(x)\leq  \frac{C}{2}|x-x_0|^2
\leq  \frac{C}{2}r^2,
$$
as desired.
\end{proof}

As shown in \cite{C77,C98,F-JEDP}, 
the upper bound is always attained.
\begin{proposition}
\label{prop:non deg}
Let $u$ solve \eqref{eq:obst},
let $x_0 \in \partial\{u>0\}$, and assume that $B_{r}(x_0)\subset \Omega$.
Then there exists  a dimensional constant $c=c(n)>0$ such that
$$
\sup_{B_r(x_0)}u\geq c\,r^2.
$$
\end{proposition}

\subsection{Blow-up analysis and Caffarelli's dichotomy}
Thanks to Corollary \ref{cor:C11} and Proposition \ref{prop:non deg}, we know that
$$
\sup_{B_r(x_0)}u\simeq r^2 \qquad \forall\, x_0\in \partial\{u>0\}.
$$
This suggests the following rescaling:
for $x_0\in \partial\{u>0\}$ and $r>0$ small, we define the family of functions
\begin{equation}
\label{eq:blow family}
u_{x_0,r}(x):=\frac{u(x_0+rx)}{r^2}.
\end{equation}
In this way, recalling Theorem \ref{thm:C11}, we get (see Figure \ref{Pic06-7}):
\begin{enumerate}
\item[$\bullet$] $u_{x_0,r}(0)=0,\qquad \sup_{B_1}u_{x_0,r}\simeq  1;
$
\item[$\bullet$]
$
|D^2u_{x_0,r}|(x)=|D^2u|(x_0+rx)\leq C.
$
\end{enumerate}
 \begin{figure}[ht]
\includegraphics[scale=0.25]{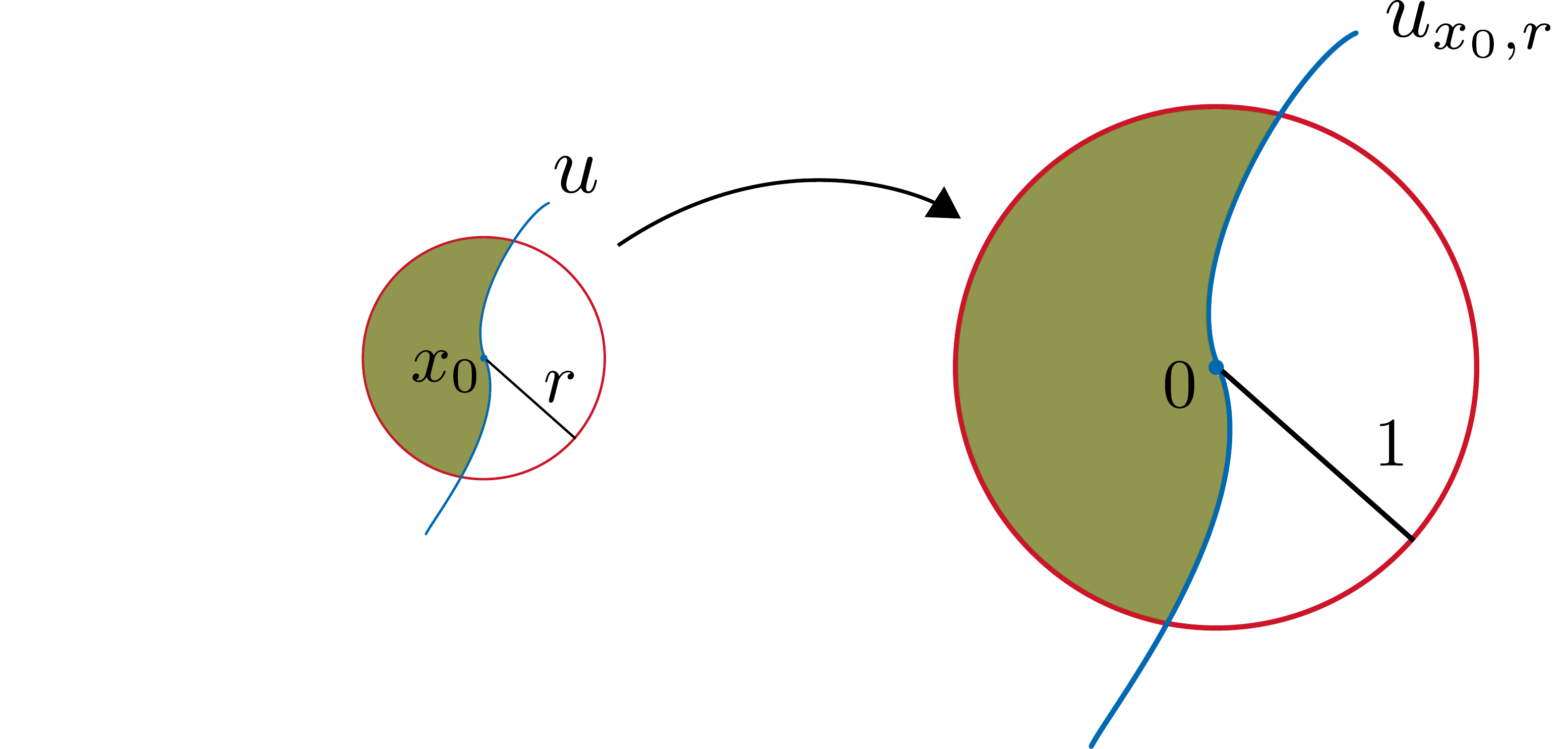}
\caption{By scaling, we look at functions of size $1$ defined inside $B_1$.} 
\label{Pic06-7}
\end{figure}
Thanks to these bounds, it follows by Ascoli-Arzel\`a Theorem that the family of functions  $\{u_{x_0,r}\}_{r>0}$
is compact in $C^1$. So, one can consider a possible limit (up to subsequences) as $r\to 0^+.$ Such a limit is called a {\em blow-up of $u$ at $x_0$}.

The first goal is to classify the possible blow-ups, since they give us information on the infinitesimal behavior of $u$ near $x_0$.
We begin by considering two possible natural type of blow-ups that one may find.

\subsubsection{Regular free boundary points}
Let us first consider the case when the free boundary is smooth near $x_0$, with $u>0$ on one side and $u=0$ on the other side.
In this case, as we rescale $u$ around $x_0$ we expect in the limit to see a one dimensional ``half-parabola'', see Figure \ref{Fig-blow1}.

\begin{figure}[ht]
\includegraphics[width=0.3\textwidth]{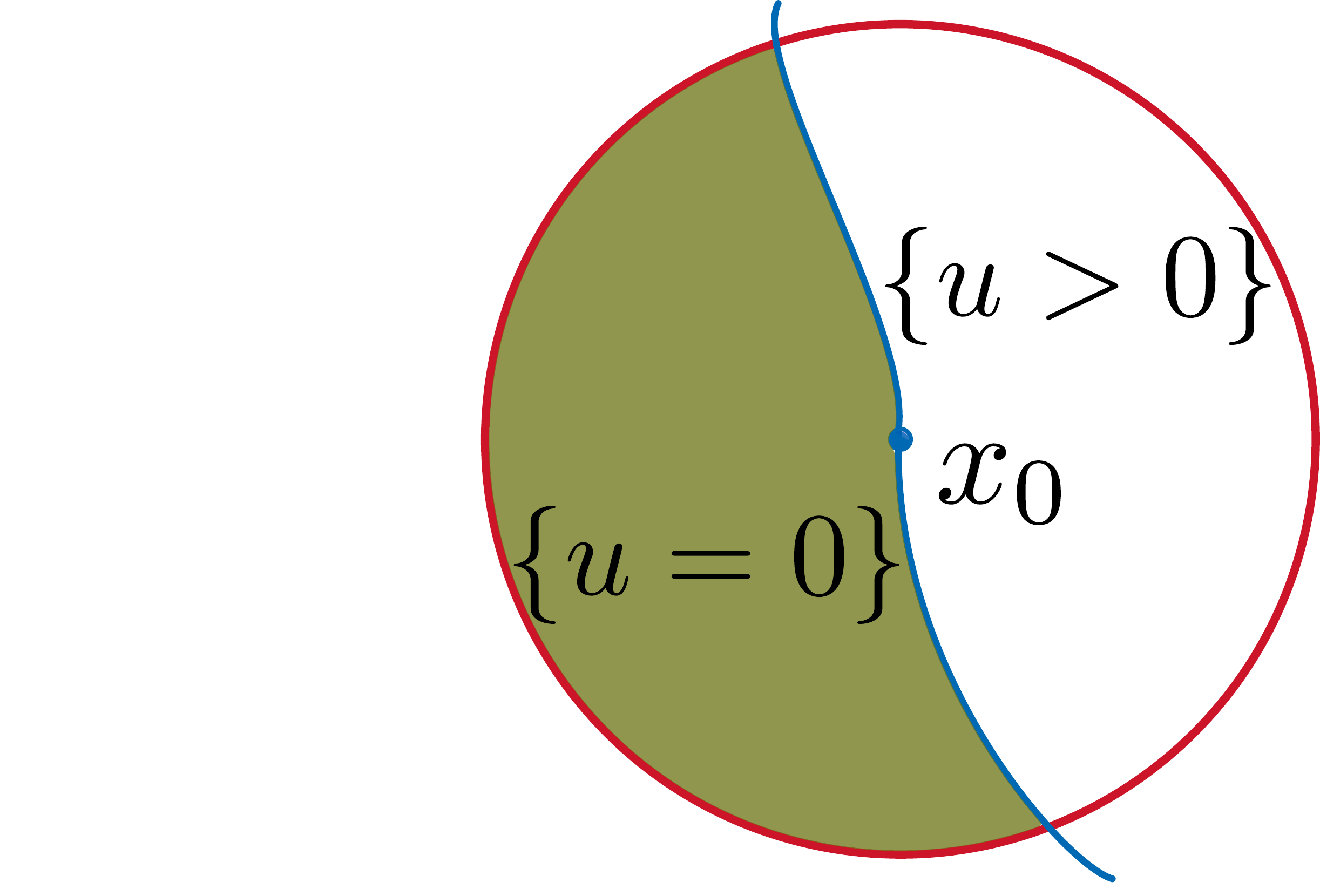}
\hspace{0.6cm}
\includegraphics[width=0.28\textwidth]{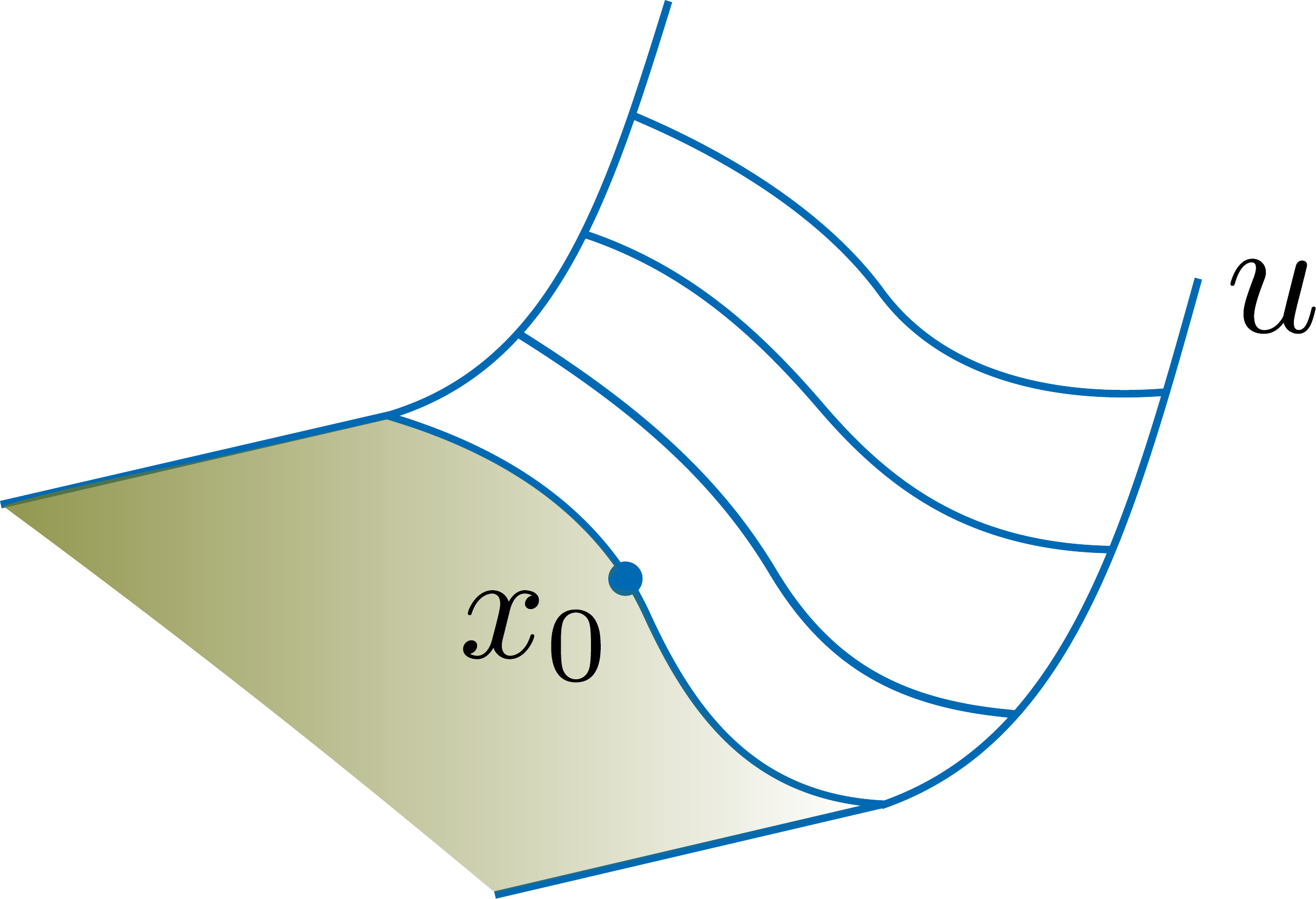}
\hspace{0.4cm}
\includegraphics[width=0.29\textwidth]{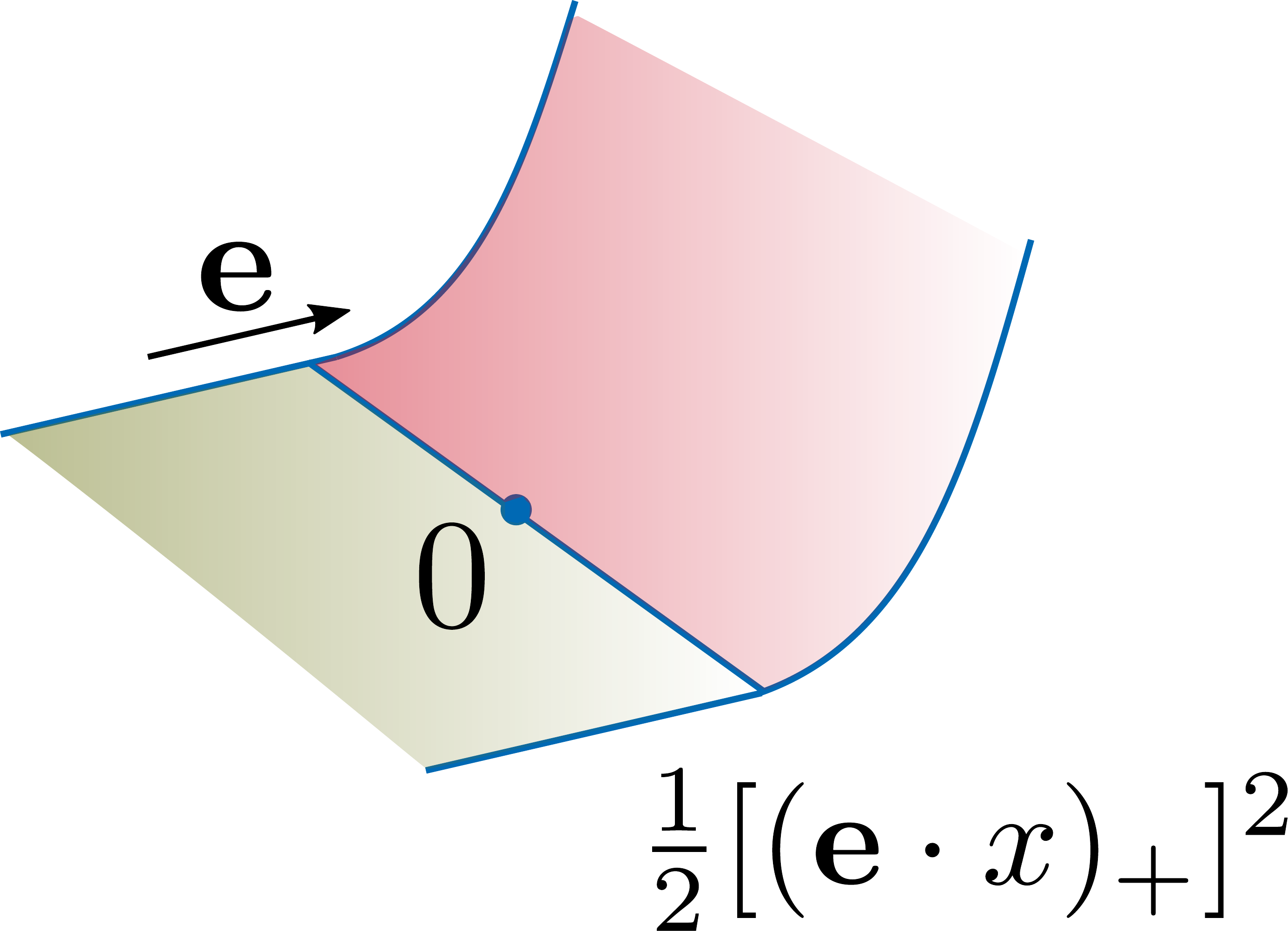}
\caption{Performing a blow-up near a ``thick'' free boundary point.}
\label{Fig-blow1}
\end{figure}

This motivates the following:
\begin{definition}
\label{def:reg pt}
A free boundary point $x_0\in \partial\{u>0\}$ is called a {\it regular point} if, up to a subsequence of radii,
$$
\frac{u(x_0+rx)}{r^2}\to \frac12 [(\mathbf{e}\cdot x)_+]^2\qquad \text{as }r\to 0^+
$$
for some unit vector $\mathbf{e}\in \mathbb S^{n-1}$.
\end{definition}

\subsubsection{Singular free boundary points}
Now, imagine that the contact-set is very narrow near $x_0$. Since $\Delta u=1$ outside of the contact set,
as
we rescale $u$ around $x_0$ we expect to see in the limit a function that has Laplacian equal to $1$ everywhere.
In dimension two, a natural behavior that one may expect to observe is represented in Figure \ref{Fig-blow2}.

\begin{figure}[ht]
\includegraphics[width=0.33\textwidth]{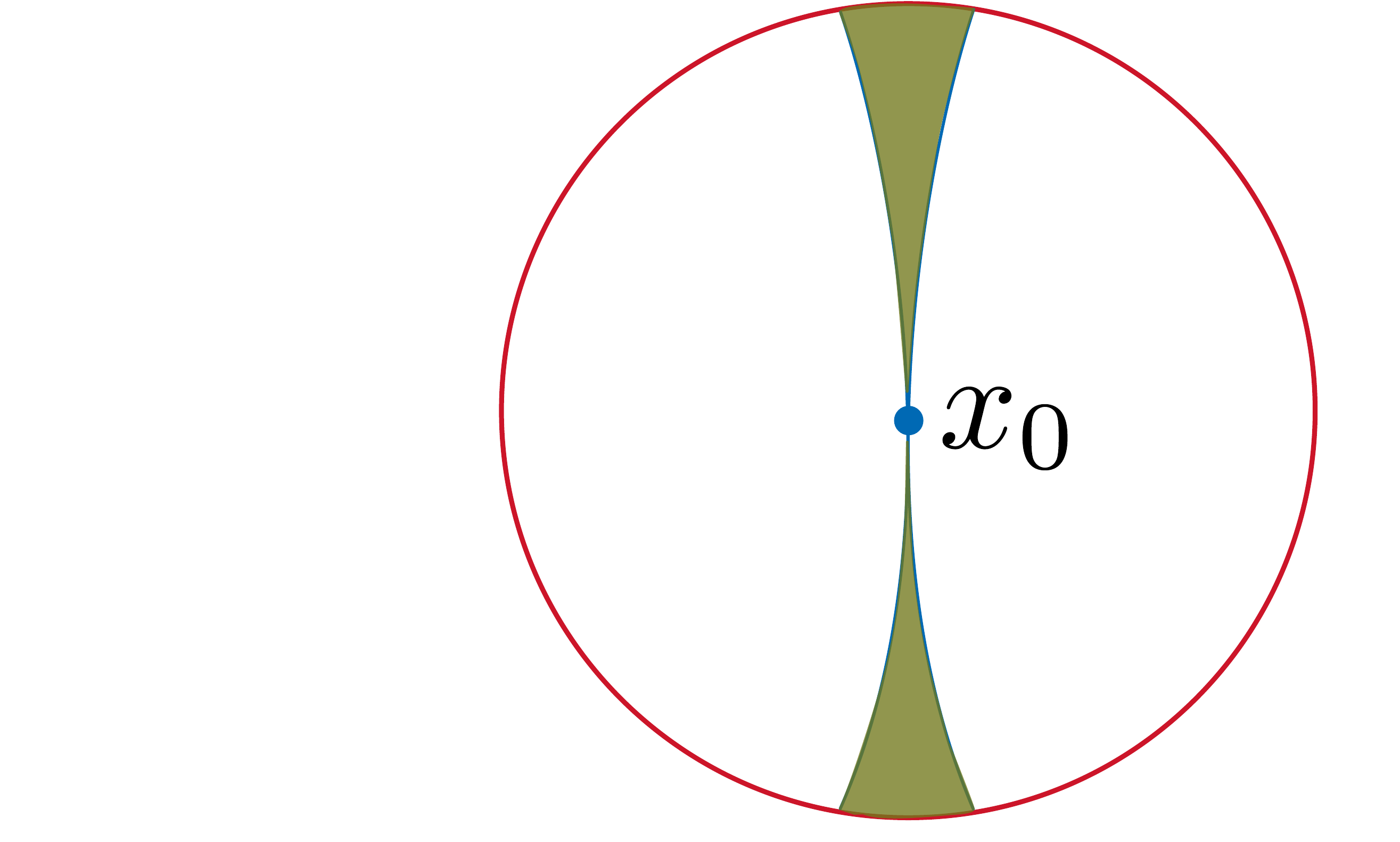}
\hspace{-0.75cm}
\includegraphics[width=0.33\textwidth]{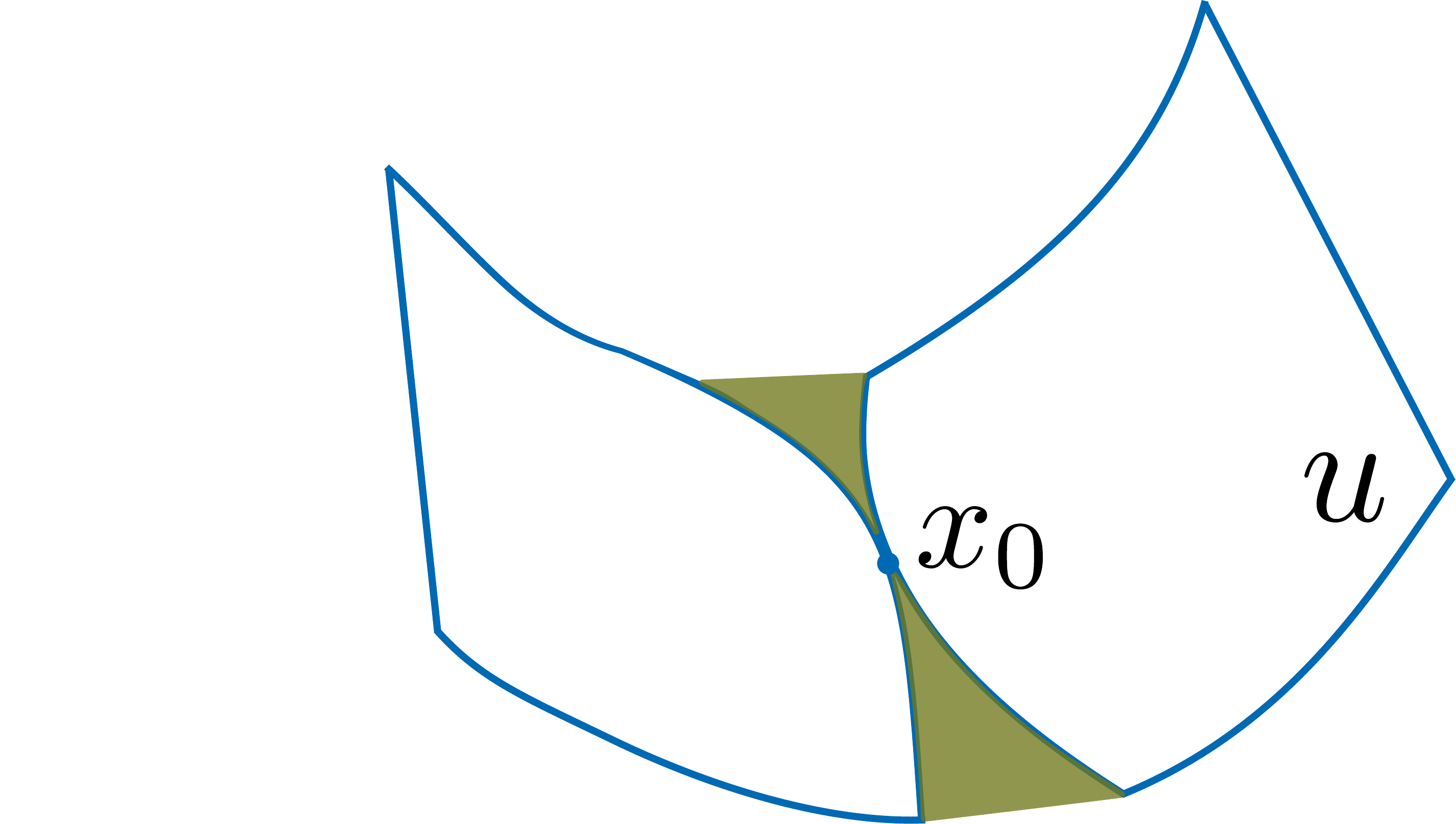}
\hspace{0.5cm}
\includegraphics[width=0.32\textwidth]{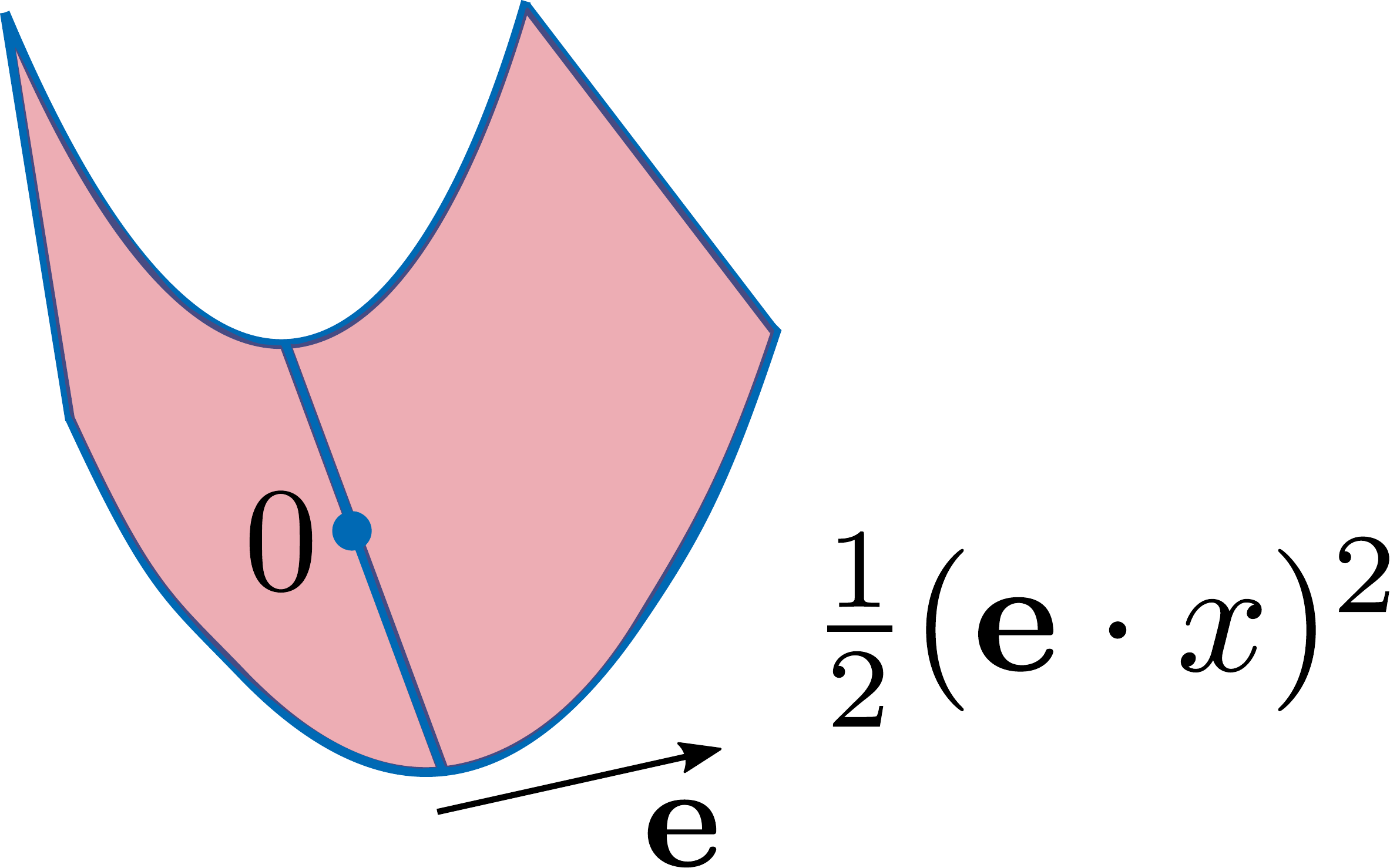}
\caption{Performing a blow-up near a ``thin'' free boundary point.}
\label{Fig-blow2}
\end{figure}
More in general, since any nonnegative quadratic polynomial with Laplacian equal to 1 solves \eqref{Pic-ell-ob},
one introduces the following:
\begin{definition}
\label{def:sing pt}
A free boundary point $x_0\in \partial\{u>0\}$ is called a {\it singular point} if, up to a subsequence of radii,
$$
\frac{u(x_0+rx)}{r^2}\to p(x):=\frac12\langle Ax,x\rangle \qquad \text{as }r\to 0^+
$$
for some nonnegative definite matrix $A \in \R^{n\times n}$ with ${\rm tr}(A)=1$.
\end{definition} 
It is worth noticing that the form of the polynomial $p$ is strictly related to the shape of the contact set near the origin.
For instance, if $n=3$ and 
$p(x)=\frac{1}2(\mathbf{e}\cdot x)^2$ for some unit  vector $\mathbf{e}\in \mathbb S^{2}$, then the contact set will be close to a 2-dimensional plane, see Figure \ref{Fig-blow2bis}.
\begin{figure}[ht]
\includegraphics[width=0.3\textwidth]{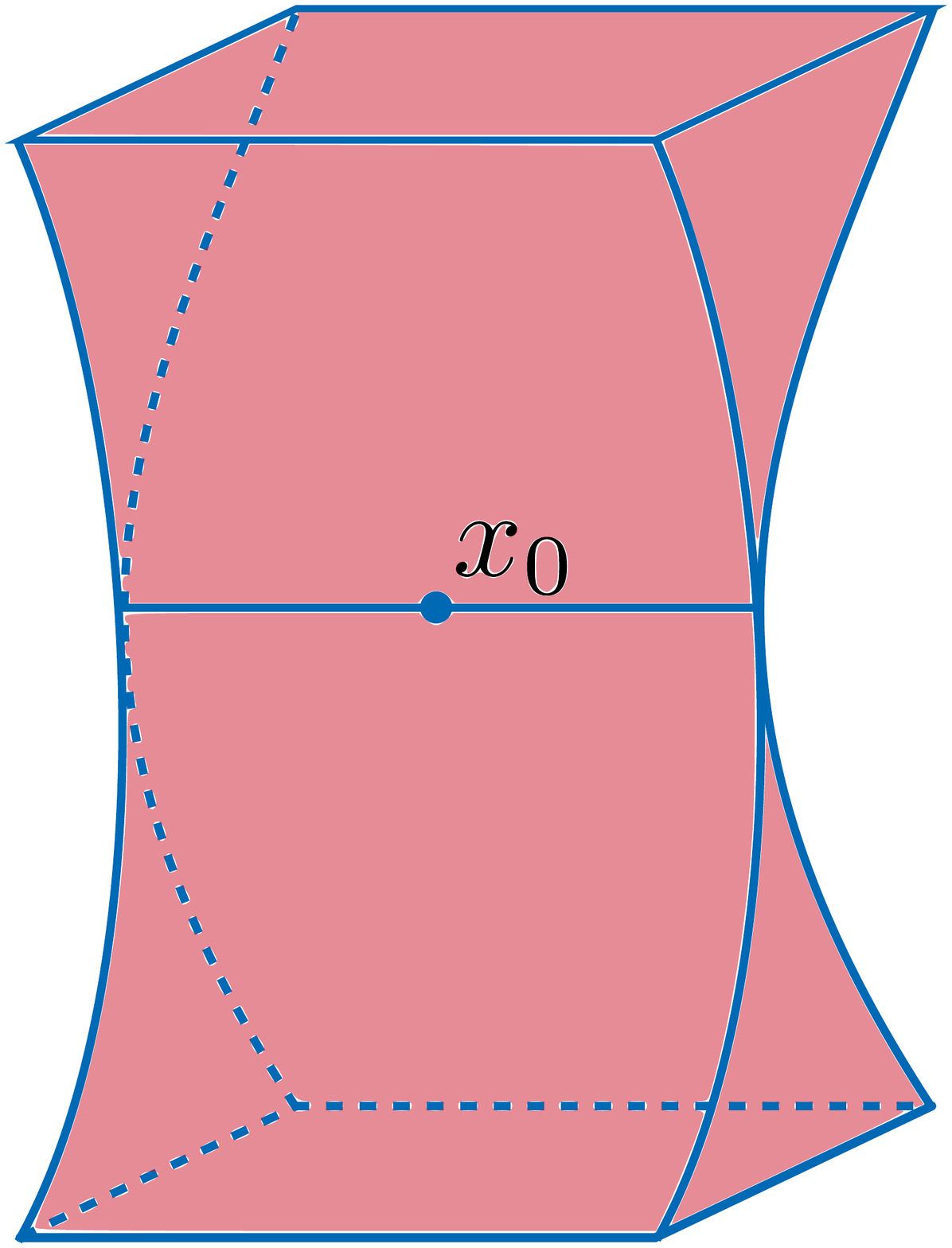}
\caption{A possible contact set near a singular point in dimension 3.}
\label{Fig-blow2bis}
\end{figure}

 On the other hand one could also see points where the contact set is close to a line, which may correspond for instance to a polynomial of the form $p(x)=\frac{1}4(x_1^2+x_2^2)$, see Figure \ref{Fig-blow2bis3}.
\begin{figure}[ht]
\includegraphics[width=0.12\textwidth]{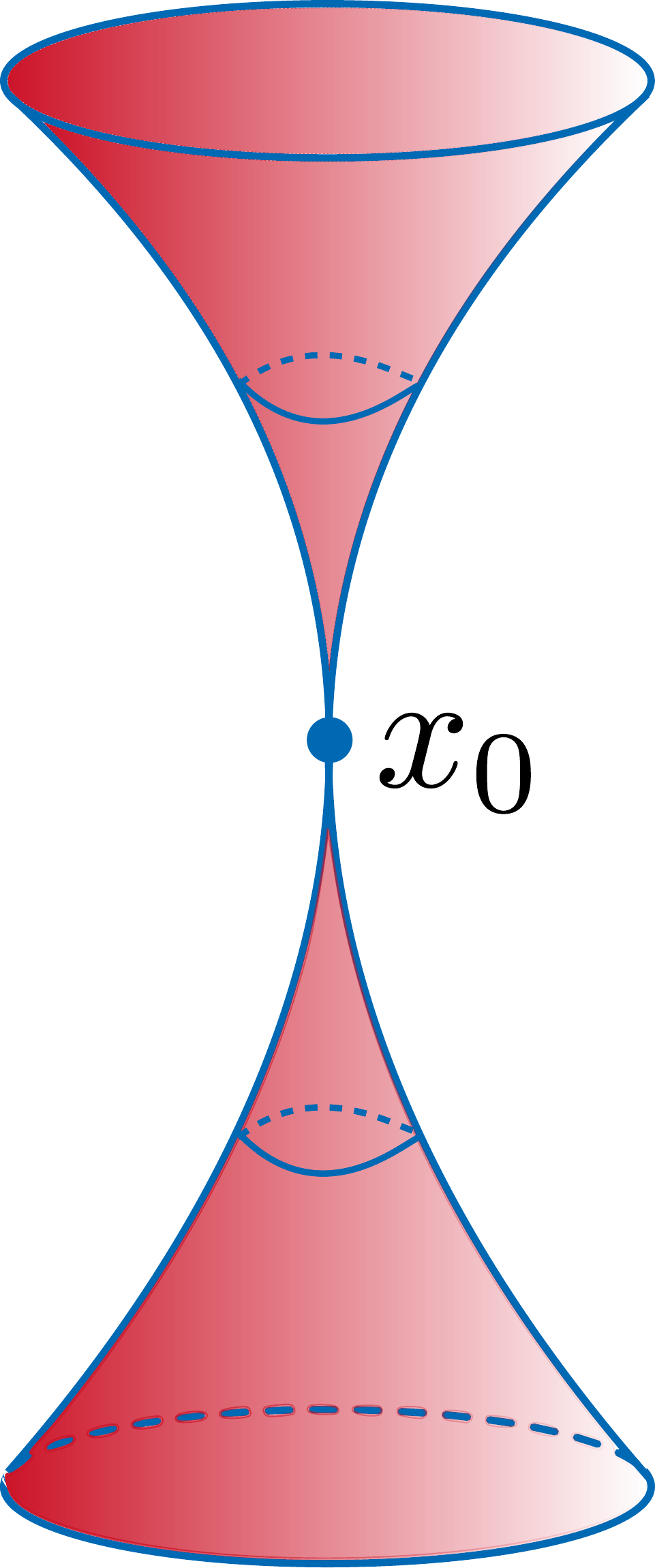}
\caption{A possible contact set near a singular point in dimension 3.}
\label{Fig-blow2bis3}
\end{figure}

\subsubsection{Caffarelli's dichotomy theorem}
Notice that, at the moment, the definitions of regular and singular points may not be mutually exclusive, since a free boundary point point could be regular along some sequence of radii and singular along a different sequence. Also, it is not clear that regular and singular points should exhaust the whole free boundary.

These highly nontrivial and deep issues have been answered by Caffarelli in \cite{C77}:

\begin{theorem}
\label{thm:dico}
Let $x_0 \in \partial\{u>0\}$. 
Then one of these two alternatives hold (see Figure \ref{Fig-reg-sing}):
\begin{enumerate}
\item[(i)]
either $x_0$ is regular, and then there exists a radius $r_0>0$ such that $\partial\{u>0\}\cap B_{r_0}(x_0)$ is an analytic hypersurface consisting only of regular points;
\item[(ii)]
or $x_0$ is singular, in which case for any $r>0$ there exists a unit vector $\mathbf e_r \in \mathbb S^{n-1}$ such that $$\partial\{u>0\}\cap B_{r}(x_0)\subset \bigl\{x\,:\,|\mathbf{e}_r\cdot (x-x_0)|\leq o(r)\bigr\}.$$
\end{enumerate}
\end{theorem}

\begin{figure}[ht]
\includegraphics[width=0.31\textwidth]{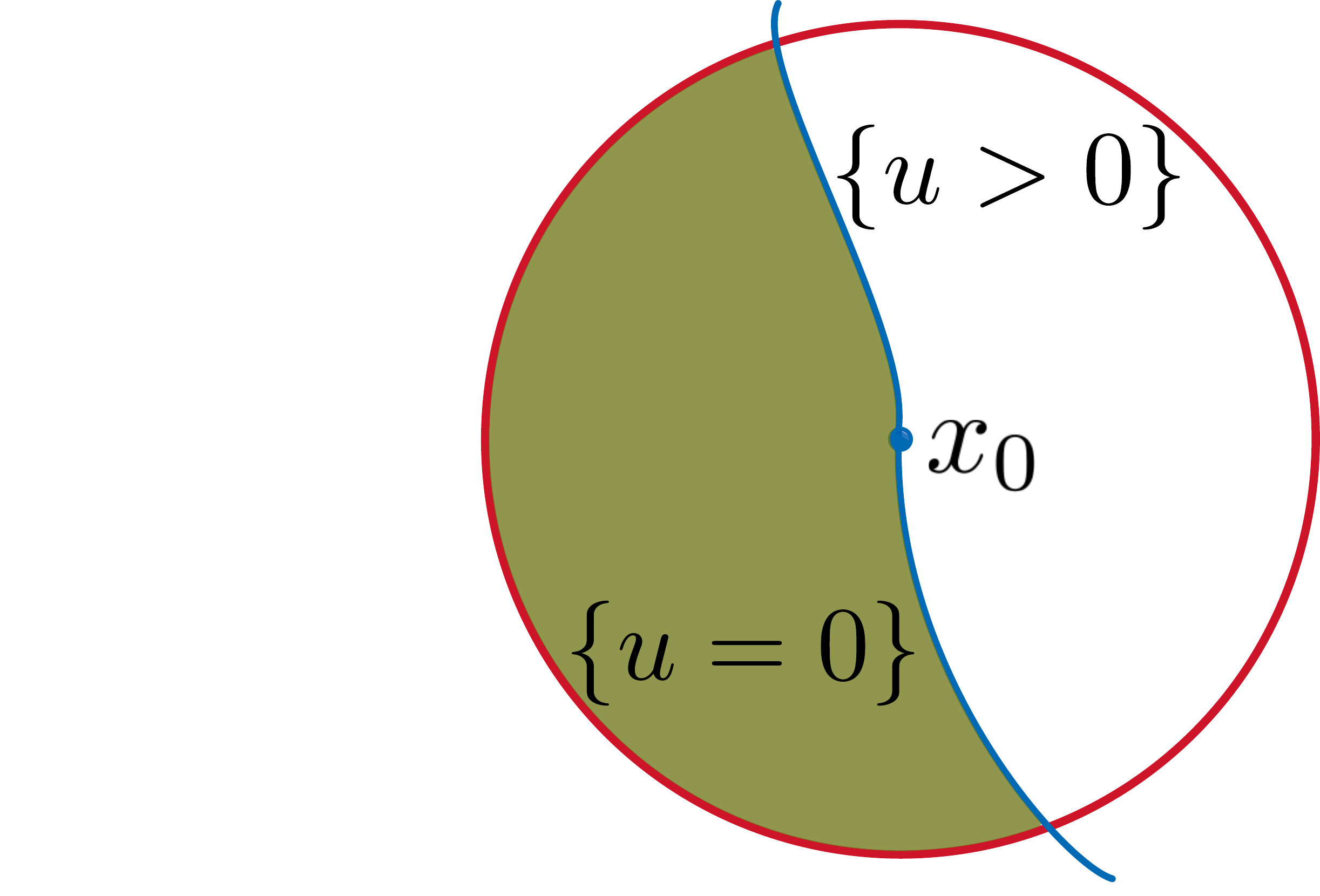}
\hspace{0.8cm}
\includegraphics[width=0.33\textwidth]{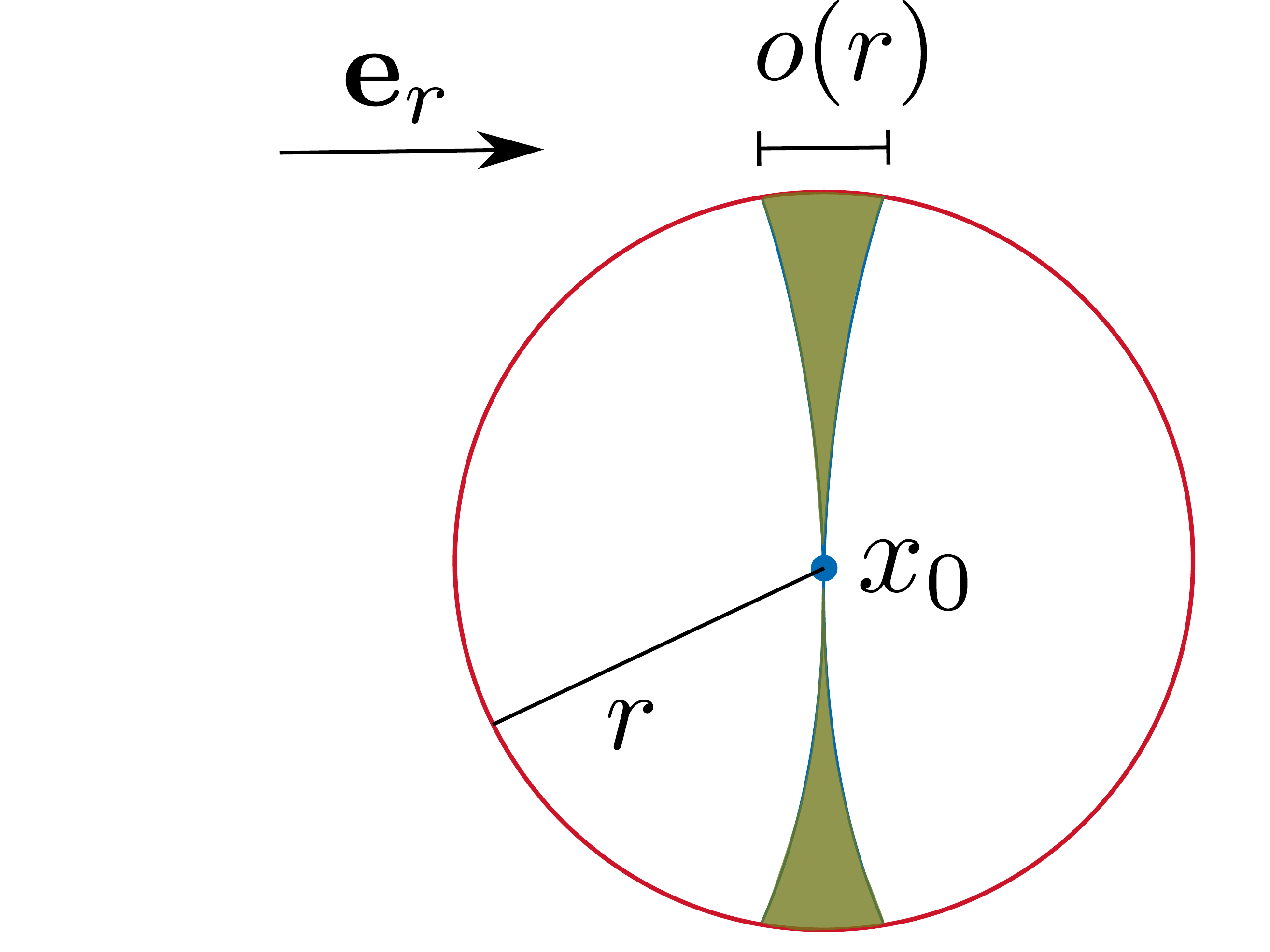}
\caption{A regular (left) and a singular (right) free boundary point.}
\label{Fig-reg-sing}
\end{figure}

Theorem \ref{thm:dico} states that a free boundary point can be either regular or singular. Also, if it is regular then the free boundary is smooth in a neighborhood and all points nearby are regular as well.
From this one deduces that the convergence in Definition \ref{def:reg pt} holds without the need of taking a subsequence of radii.

While Theorem \ref{thm:dico}(i) gives a  complete answer on the structure of regular points, Theorem \ref{thm:dico}(ii) is still not conclusive. Indeed,  in the statement the vector $\mathbf{e}_r$ may depend on $r$. Also, the quantity $o(r)$ comes from a compactness argument, so it is not quantified.

Hence, from now on we shall focus on the study of singular points.
To simplify the notation, we  denote
$$\Sigma:=\{\text{singular points}\}\subset \partial\{u>0\}.$$
Note that, since the set of regular points is relatively open inside the free boundary (see Theorem \ref{thm:dico}(i)), it follows that $\Sigma$ is a closed set.

\subsection{Uniqueness of blow-up at singular points}

As observed in the previous section, a priori the vector $\mathbf{e}_r$ appearing in the statement of
 Theorem \ref{thm:dico}(ii)
may depend on $r$. This fact is essentially related to the question of whether the convergence in Definition \ref{def:sing pt} holds up to  subsequences or not: indeed, if one could prove that the convergence to a polynomial $p$ holds without passing to a subsequence, then one could easily deduce that 
$$\partial\{u>0\}\cap B_{r}(x_0)\subset \bigl\{x\,:\, {\rm dist}(x-x_0,\{p=0\})\leq o(r)\bigr\}.$$

The complete answer to this questions has  been given by Caffarelli in \cite{C98}, after important previous results in the two dimensional case \cite{CR77,Sak91,Sak93}.

From now on, we use the notation 
\begin{multline*}
\mathcal P := \bigg\{ p(x) = { \frac 1 2} \langle  Ax,x\rangle\  :\ \\   A\in \R^{n\times n}\mbox{ symmetric nonnegative definite}, \mbox{ ${\rm tr}\, A=1$}
\bigg\}
.
\end{multline*}

\begin{theorem}
\label{thm:uniq blow 2}
Let $x_0\in \Sigma$. Then
there exists $p_{*,x_0} \in \mathcal P$ such that
$$
\lim_{r\to 0}\frac{u(x_0+rx)}{r^2}= p_{*,x_0}(x).
$$
In addition, the map
$$
\Sigma\ni x_0\mapsto p_{*,x_0}
$$
is locally uniformly continuous.
\end{theorem}
We present here a proof of this result given few years later by Monneau \cite{M03}. To this aim, we first recall the
so-called Monneau's monotonicity formula, whose proof relies on a previous monotonicity formula obtained by Weiss in \cite{W99}.
\begin{lemma}\label{lem:Mon}
Let $0 \in \Sigma$, $p \in \mathcal P$, and define
$$
M(r,u,p):=\frac{1}{r^{n+3}}
 \int_{\partial B_r} (u-p)^2.
$$
Then
$$
\frac{d}{dr}M(r,u,p)\geq 0.
$$
\end{lemma}
Using this lemma, the uniqueness and the continuity of the blow-up at singular points follows rather easily.
\begin{proof}[Proof of Theorem \ref{thm:uniq blow 2}]
We first prove the existence of the limit.

Assume with no loss of generality that $x_0=0$,
and set $$u_r(x):=r^{-2}u(rx).$$
With this notation, noticing that $r^{-2}p(rx)=p(x)$, it is follows by a change of variables that
\begin{equation}
\label{eq:change M}
M(r,u,p)=
\frac{1}{r^{n+3}}
 \int_{\partial B_r} (u-p)^2
 =
 \int_{\partial B_1} (u_r-p)^2.
\end{equation}
Now, let 
$p_1$ and $p_2$ be two different limits obtained along two sequences $r_{k,1}$ and $r_{k,2}$ both converging to zero. Up to taking a subsequence of $r_{k,2}$ and relabeling the indices, we can assume that $r_{k,2}\leq r_{k,1}$ for all $k$.
Thus, thanks to Lemma \ref{lem:Mon}
and \eqref{eq:change M}, we have
$$\int_{B_1}(u_{r_{k,1}}-p_1)^2=M(r_{k,1},u,p_1)\geq M(r_{k,2},u,p_1)=\int_{B_1}(u_{r_{k,2}}-p_1)^2, 
$$
and letting $k\to \infty$ we obtain
$$
0=\lim_{k\to \infty} \int_{B_1}(u_{r_{k,1}}-p_1)^2\geq \lim_{k\to \infty}\int_{B_1}(u_{r_{k,2}}-p_1)^2 =\int_{B_1}(p_2-p_1)^2.
$$
This proves that there is a unique possible limit for the functions $u_r$ as $r\to 0$, which implies that the limit exists.
From now on, given a singular point $x_0$, we shall denote this limit by $p_{*,x_0}$.

We now prove the continuity of the map $x_0 \mapsto p_{*,x_0}$ at $0 \in \Sigma$.
Fix $\ep>0$, and consider a sequence $x_k\in \Sigma$ with $x_k\to 0$.
Since $u_r \to p_{*,0}$ as $r\to 0$, there exists a small radius $r_\ep>0$ such that
\begin{equation}
\label{eq:eps}
\int_{\partial B_1}\biggl|\frac{u(r_\ep x)}{r_\ep^2}-p_{*,0}(x)\biggr|^2 \leq \ep.
\end{equation}
Then, applying Lemma \ref{lem:Mon} at $x_k$ with $p=p_{*,0}$, we deduce that
\begin{align*}
\int_{\partial B_1}|p_{x_k,*} - p_{*,0}|^2 &=\lim_{r\to 0}
 \int_{\partial B_1}\biggl|\frac{u(x_k+r x)}{r^2} - p_{*,0}(x)\bigg|^2\\
 & \leq 
  \int_{\partial B_1}\biggl|\frac{u(x_k+r_\ep x)}{r_\ep^2} - p_{*,0}(x)\bigg|^2.
\end{align*}
Hence, letting $k \to \infty$ and recalling \eqref{eq:eps} we obtain
\begin{align*}
\limsup_{k\to \infty}\int_{\partial B_1}
|p_{x_k,*} - p_{*,0}|^2\leq \lim_{k\to \infty} \int_{\partial B_1}\biggl|\frac{u(x_k+r_\ep x)}{r_\ep^2} - p_{*,0}(x)\bigg|^2
 \le	 \ep.
\end{align*}
Since $\ep>0$ is arbitrary, this proves the continuity at $0$.

Because $\Sigma$ is locally compact (recall that $\Sigma$ is closed), this actually implies that the map
$$
\Sigma\ni x_0\mapsto p_{*,x_0}
$$
is locally uniformly continuous.
\end{proof}

\subsection{Stratification and $C^1$ regularity of the singular set}
With Theorem \ref{thm:uniq blow 2} at hand, we can now investigate the regularity of $\Sigma$.
Note that singular points may look very different, depending on the dimension of the set $\{p_{*,x_0}=0\}$, see Figures 
\ref{Fig-blow2bis} and
\ref{Fig-blow2bis3}.
This suggests to stratify the set of singular points according to this dimension.
More precisely,
given
$x_0\in \Sigma$ we set
$$
k_{x_0}:={\rm dim}({\rm ker} \,D^2p_{*,x_0})
={\rm dim}(\{p_{*,x_0}=0\}).
$$
Then, given $m\in \{0,\ldots,n-1\}$, we define
$$
\Sigma_m:=\{x_0\in \Sigma\,:\,k_{x_0}=m\}.
$$
Note that, with this definition, the point in Figure \ref{Fig-blow2bis} belongs to $\Sigma_2$, while the point in Figure \ref{Fig-blow2bis3} belongs to $\Sigma_1$.
Hence one may expect that $\Sigma_m$ should correspond to the $m$-dimensional part of $\Sigma$, see Figure \ref{Fig-contact}.
\begin{figure}[ht]
\includegraphics[width=0.7\textwidth]{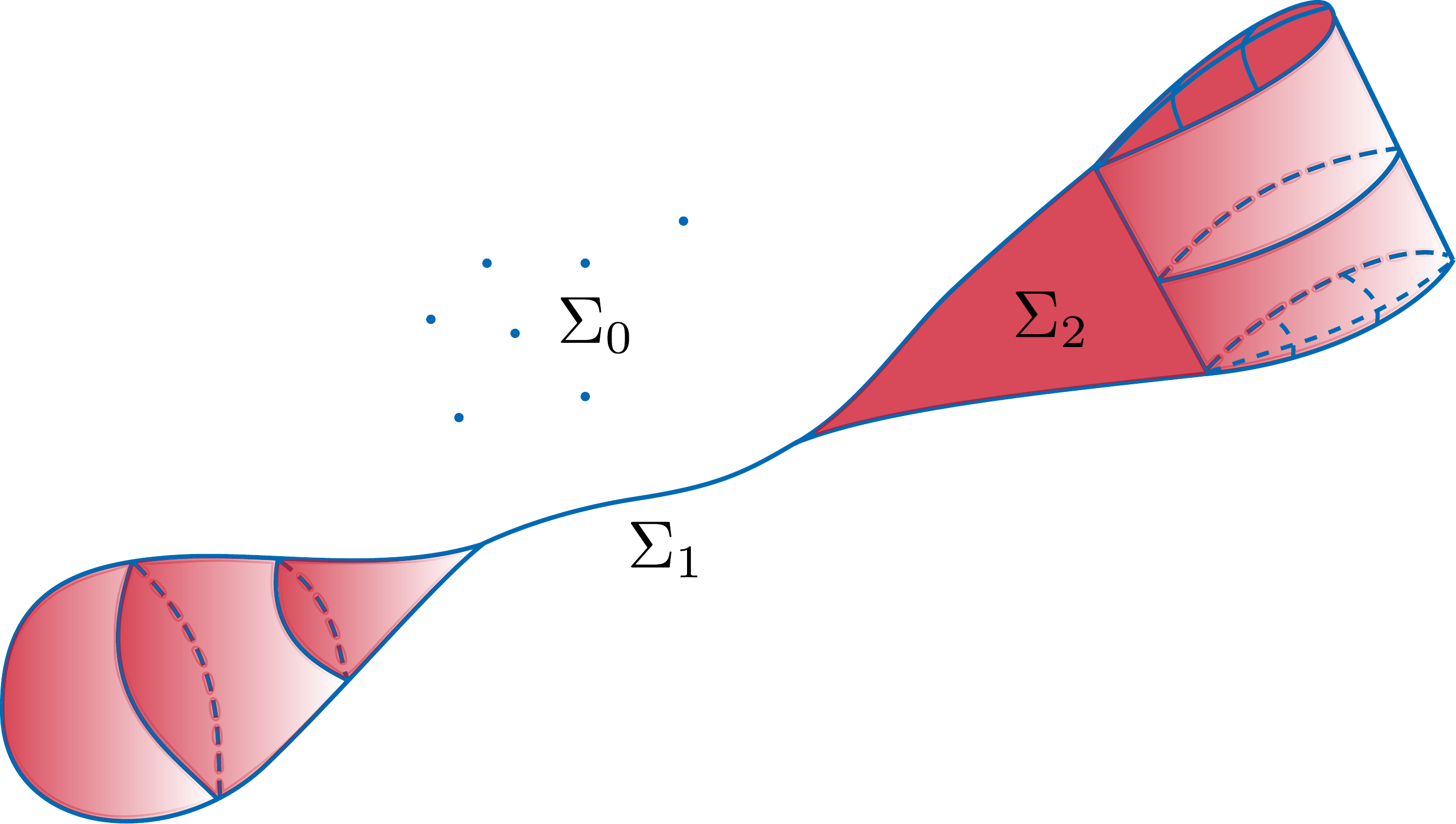}
\caption{A possible example of contact set in 3 dimensions.}
\label{Fig-contact}
\end{figure}

This intuition is confirmed by the following result of Caffarelli \cite{C98}:
\begin{theorem}
\label{thm:C1}
For any  $m\in \{0,\ldots,n-1\}$,
$\Sigma_m$ is locally contained in 
a $m$-dimensional manifold of class $C^1$.
\end{theorem}

\begin{proof}[Idea of the proof]
Recalling
\eqref{eq:zero grad}, we have
 $u|_{\Sigma_m}=\nabla u|_{\Sigma_m}\equiv 0$.
Also, thanks to Theorem \ref{thm:uniq blow 2},
$$
u(x_0+y)=p_{*,x_0}(y)+o(|y|^2).
$$
Hence, at least formally, $p_{*,x_0}$ corresponds to the second order term in the Taylor expansion of $u$,
namely ``$p_{*,x_0}(y)=\frac12 \langle D^2u(x_0)y,y\rangle$'', or equivalently
$$
\text{``$D^2 p_{*,x_0}=D^2u(x_0)$''.}
$$
Since the map
$$\Sigma\ni x_0\mapsto p_{*,x_0}(y)=\frac12 \langle D^2p_{*,x_0}y,y\rangle$$ is continuous, we deduce that
$$
\Sigma\ni x_0\mapsto D^2 p_{*,x_0}\in \R^{n\times n}
$$
is continuous as well.

This allows us to apply Whitney's extension theorem to find a map $F:\R^n\to \R^n$ of class $C^1$ such that
$$
F(x_0)=\nabla u(x_0)=0 
\qquad \text{and}\qquad \nabla F(x_0)=D^2p_{*,x_0}\qquad \forall\,x_0\in \Sigma_m.
$$
Noticing that, by the definition of $\Sigma_m$,
$$
{\rm dim}({\rm ker }\,\nabla F(x_0))={\rm dim}({\rm ker }\,D^2p_{*,x_0})=m\qquad \text{ on $\Sigma_m$},
$$
it follows by the Implicit Function Theorem that
$$
{\Sigma_m}=\{F=0\}\cap {\Sigma_m}
$$
is locally contained in 
a $C^1$ $m$-dimensional manifold,
 as desired.
%
%
%
\end{proof}

\begin{remark}
The proof above shows that the estimate
$$
\|u(x_0+\cdot)-p_{*,x_0}\|_{L^\infty(B_r)}= o(r^2),
$$
with an error $o(r^2)$ independent of $x_0$, implies that 
$\Sigma_m$ is locally contained in a $C^1$ $m$-dimensional manifold. 

More in general, if one could prove that
\begin{equation}
\label{eq:C2a u}
\|u(x_0+\cdot)-p_{*,x_0}\|_{L^\infty(B_r)}\leq C\,r^{2+\alpha}
\end{equation}
for some constant $C$ independent of $x_0$, then by applying Whitney's extension theorem in H\"older spaces one would conclude that $\Sigma_m$ is contained in a $m$-dimensional manifold of class $C^{1,\alpha}$.
\end{remark}

\begin{remark}
The fact that $\Sigma_m$ is only contained in a manifold (and does not necessarily coincide with it) is optimal: already for $n=2$, one can build examples where $\Sigma_1$ coincides with a Cantor set contained in a line \cite{Sch76}.
\end{remark}

\subsection{Recent developments}
\label{sect:recent}

In 1999,  Weiss proved a monotonicity formula that allowed him to obtain the following result \cite{W99}:
\begin{theorem}
\label{thm:Weiss}
Let $n=2$. Then there exist $C,\alpha>0$ such that
$$
\|u(x_0+\cdot)-p_{*,x_0}\|_{L^\infty(B_r)}
\leq
C \,
r^{2+\alpha}\qquad \forall\,x_0 \in \Sigma.
$$
In particular $
\Sigma_1$ is  locally contained in a $C^{1,\alpha}$ curve.
\end{theorem}

Weiss' proof was restricted to two dimensions because of some delicate technical assumptions in some steps of the proof. Still, one could have hoped to extend his argument to higher dimensions. This was achieved by Colombo, Spolaor, and Velichkov \cite{CSV17}. There, the authors introduced a quantitative argument to avoid a compactness step in Weiss' proof. However, the price to pay for working in higher dimensions was that they could only get a logarithmic improvement in the convergence of $u$ to $p_{*,x_0}$:
\begin{theorem}
\label{thm:CSV}
Let $n\geq 3$. Then
exist dimensional constants $C,\epsilon>0$ such that
$$
\|u(x_0+\cdot)-p_{*,x_0}\|_{L^\infty(B_r)}
\leq
C \,r^2|\log(r)|^{-\epsilon}\qquad \forall\, x_0 \in \Sigma.
$$
In particular, for any $m \in \{0,\ldots,n-1\}$, $
\Sigma_m$  is  locally contained in a $C^{1,\log^\epsilon}$ $m$-dimensional manifold.
\end{theorem}
In other words, in dimension $n\geq 3$ one can improve the $C^1$ regularity of Caffarelli to a quantitative one, with a logarithmic modulus of continuity.
 This result raises the question of whether one may hope to improve such an estimate, or if this logarithmic bound is optimal.

 \smallskip

In a recent paper with Serra \cite{FSell} we showed that, at most points, \eqref{eq:C2a u} holds with $\alpha=1$. However, there exist some ``anomalous'' points of
higher codimension where not only  \eqref{eq:C2a u} does not hold with $\alpha=1$, but actually \eqref{eq:C2a u}  is false for any $\alpha>0$.

As a consequence we deduce that, up to a small set, 
singular points can be covered by $C^{1,1}$ (and in some cases $C^2$) manifolds. As we shall discuss in Remark \ref{rmk:optimal} below, this result provides the optimal decay estimate for the contact set.

Finally, it is important to observe that anomalous points may exist and our bound on their Hausdorff dimension is optimal.

Before stating our result we note that, as a consequence of Theorem \ref{thm:uniq blow 2}, points in $\Sigma_0$ are isolated and $u$ is strictly positive in a neighborhood of them. In particular $u$ solves $\Delta u=1$ in a neighborhood of $\Sigma_0$, hence it is analytic there. Thus, it is enough to focus on the cases $m=1,\ldots,n-1$.

Here and in the sequel, ${\rm dim}_{\mathcal H}(E)$ denotes the Hausdorff dimension of a set $E$.
The main result in \cite{FSell} the following:
\begin{theorem}
\label{thm:main}
Let $\Sigma:=\cup_{m=0}^{n-1}\Sigma_m$ denote the set of singular points. Then:
\begin{enumerate}
\item[($n=2$)] $\Sigma_1$ is locally  contained in a $C^{2}$ curve.
\item[($n\ge 3$)] \begin{enumerate}
\item
The higher dimensional stratum $\Sigma_{n-1}$ 
can be written as the disjoint union of ``generic points'' $\Sigma_{n-1}^g$ and ``anomalous points'' $\Sigma_{n-1}^a$, where:\\
- $\Sigma^g_{n-1}$ is  locally  contained in a $C^{1,1}$ $(n-1)$-dimensional manifold;\\
- $\Sigma^a_{n-1}$ is a relatively  open subset of $\Sigma_{n-1}$ satisfying  $${\rm dim}_{\mathcal H}(\Sigma^a_{n-1})\leq n-3$$  (actually, $\Sigma^a_{n-1}$ is discrete when $n=3$).

Furthermore, the whole stratum $\Sigma_{n-1}$ can be  locally  covered by a $C^{1,\alpha_\circ}$  $(n-1)$-dimensional manifold, for some dimensional exponent $\alpha_\circ>0$. 
\item For all $m=1,\ldots,n-2$ we can write $\Sigma_{m}=\Sigma_m^g\cup\Sigma_g^a$,
where:\\
- $\Sigma_m^g$ is locally contained in a $C^{1,1}$
$m$-dimensional manifold;\\
- $\Sigma^a_m$ is a relatively open subset of $\Sigma_{m}$ satisfying $${\rm dim}_{\mathcal H}(\Sigma^a_m)\leq m-1$$ (actually, $\Sigma^a_m$ is discrete when $m=1$).

In addition, the whole stratum $\Sigma_{m}$ can be  locally covered   by a $C^{1,\log^{\epsilon_\circ}}$ $m$-dimensional manifold, for some dimensional exponent $\epsilon_\circ>0$. 
\end{enumerate}
\end{enumerate}
\end{theorem}

This result needs several comments.

\begin{remark}\label{rmk:optimal}
We first discuss the optimality of the theorem above, and then make some general considerations.
\begin{enumerate}
\item
Our $C^{1,1}$ regularity provides the optimal control on the contact set in terms of the density decay. Indeed our result implies that, at all singular points up to a $(n-3)$-dimensional set, the following bound holds:
$$
\frac{|\{u=0\}\cap B_r(x_0)|}{|B_r(x_0)|}\leq Cr\qquad \forall\,r>0.
$$
In view of the two dimensional Example 1 in \cite[Section 1]{Sch76}, this estimate is optimal.

\item The possible presence of anomalous points comes from different reasons depending on the dimension of the stratum. More precisely, the following holds:
\begin{enumerate}
\item The possible presence of points in $\Sigma_{n-1}^a$ comes from the potential existence, in dimension $n\ge 3,$ of $\lambda$-homogeneous solutions to the so-called Signorini problem with $\lambda\in(2,3)$, see for instance \cite{AC04,ACS08}. Whether this set is empty or not is an important open problem.
 \item The anomalous points in the strata $\Sigma^a_m$ for $m\le n-2$ come from the possibility that, around a singular point $x_0$, the function $u$ behaves as
$$
u(x_0+rx)=r^2\, p_{*,x_0}(x)
+r^2 \varepsilon_r \,q(x)+o(r^2\varepsilon_r),
$$ 
where:\\
 - $\varepsilon_r\in \R^+$ is infinitesimal as $r\to 0^+$, but $\varepsilon_r\gg r^\alpha$ for any $\alpha>0$;\\
- $q$ is a nontrivial second order harmonic polynomial.\\
This behavior may look rather strange: indeed we are saying that, after one removes  from $u$ its second order Taylor expansion $p_{*,x_0}$, one still sees as a reminder a second order polynomial.
However, it turns out that such anomalous points may exist, and  we can construct examples of solutions for which ${\rm dim}_{\mathcal H}(\Sigma_m^a)=m-1$. 
\end{enumerate}
\item Our result on the higher dimensional stratum $\Sigma_{n-1}$ extends Theorem \ref{thm:Weiss} to every dimension, and improves it in terms of the regularity. 
\item
The last part of the statement in the case ($n\geq 3$)-(b) corresponds to Theorem \ref{thm:CSV}. In \cite{FSell} we obtain the same result as a simple byproduct of our analysis. In addition, our result on the existence of anomalous points shows that Theorem \ref{thm:CSV} is essentially optimal.
 \end{enumerate}
 \end{remark}

\section{Generalizations and applications}

\subsection{The parabolic obstacle problem}
The first natural extension of Theorem \ref{thm:main} consists in understanding the structure of the free boundary in the parabolic case
 \begin{equation}
 \label{eq:parab obst2}
\partial_tu=\Delta u-\chi_{\{u>0\}},\qquad u \geq 0,\qquad \partial_tu \geq 0.
\end{equation}
As shown in \cite{C77} (see also \cite{CF79}), solutions to this problem are $C^1$ in time and $C^{1,1}$ in space, namely
$$
|\partial_t u|+|u|+|\nabla u|+|D^2u| \in L^\infty_{\rm loc},\qquad \partial_t u \in C^0.
$$
Also, as in the elliptic case, points of the {free boundary} $\partial\{u>0\}$ are divided into two classes: regular points and singular points. A free boundary point $z_0=(t_0,x_0)$ is either regular or singular depending on the type of blow-up of $u$ at that point. More precisely:
\begin{multline}\label{regular}
z_0 \mbox{ is called \emph{regular} point}  \quad \Leftrightarrow \\  \quad \frac{u(t_0+r^2t,x_0+ rx)}{r^2} \ \stackrel{r\downarrow 0}\longrightarrow\  \frac12 [(\mathbf{e}\cdot x)_+]^2
\end{multline}
for some $\boldsymbol e=\boldsymbol e_{z_0}\in \mathbb S^{n-1}$, and
\begin{multline}\label{singular}
z_0 \mbox{ is called \emph{singular} point}  \quad \Leftrightarrow\\
 \quad  \frac{u(t_0+r^2t,x_0+ rx)}{r^2} \ \stackrel{r\downarrow 0}\longrightarrow \ p_{*,z_0}(x) :=\frac 1 2 \langle  Ax,x\rangle
\end{multline}
for some symmetric nonnegative definite matrix $A=A_{z_0}\in \R^{n\times n}$ with  ${\rm tr}(A)=1$. 
The existence of the previous limits in \eqref{regular} and \eqref{singular},
as well as the classification of possible blow-ups are well-known results; see \cite{C77,CPS00,B06}.
It is interesting to observe that both at regular and singular points the blow-ups are independent of time. This can be explained as follows:
since solutions are $C^1$ in time and $C^{1,1}$ in space and $\partial_t u=\nabla u=0$ on the free boundary, near a free boundary point $z_0=(t_0,x_0)$ the function $u$ satisfies
$$u(t_0+t,x_0+x)=o(t)+O(|x|^2).$$ Because of this, it follows immediately that the blow-ups considered above will not depend on $t$.

By the theory in \cite{C77,CPS00}, the free boundary is an analytic hypersurface near regular points. On the other hand, near singular points the {contact set} $\{u=0\}$ may form cusps and can be rather complicated.

To understand the structure of the singular points, we consider again a stratification based on the size of the zero set of the blow-up.
More precisely, given a singular point $z_0=(t_0,x_0)$, we set
\[L_{z_0} := \{ p_{*,z_0} =0 \} = {\rm ker}(A_{z_0}).\]
Then, given a time $t>0$ and $m\in \{0,1,2,\dots, n-1\}$, we define the $m$-th stratum at time $t$ as
\[
\Sigma_{m,t} := \big\{z_0=(t_0,x_0) \ : \mbox{ singular point with } \dim( L_{z_0}) =m \text{ and }t_0=t\big\}.
\]
The natural generalization of Theorem \ref{thm:C1} would be to prove that, for any $t$ and $m$, the set $\Sigma_{m,t}$ is contained in a $m$-dimensional manifold of class $C^1$.
Actually, as proved in \cite{LM15} (see also \cite{BDM06,B06} for some previous contributions)
$$
\bigcup_{t>0} \bigl(\{t\}\times \Sigma_{m,t}\bigr)\subset \R^+\times \R^n
$$
is locally contained in a $m$-dimensional manifold of class $C^1$,
where here $C^1$ regularity has to be intended with respect to the parabolic metric
$$
d_{P}(z,z'):=|x-x'|+|t-t'|^{1/2},\qquad \text{where }z=(x,t),\,z'=(x',t').
$$
In our paper \cite{FRS} we take a different approach. More precisely,
 because the function $u$ is mononically increasing in time, the free boundaries $\partial\{u(t)>0\}$ are nested.
Exploiting this monotonicity one could use the arguments in \cite{LM15} to show that
$$
\bigcup_{t>0}\Sigma_{m,t}\subset \R^n
$$
is locally contained in a $m$-dimensional manifold of class $C^1$.
Hence, it is natural to try to prove the analogue of Theorem \ref{thm:main} for $\cup_{t>0}\Sigma_{m,t}$.
This is done in \cite{FRS}, where we obtain the following result (we state here a simplified version):
\begin{theorem}
\label{thm:main2}
Let $u$ be a solution of \eqref{eq:parab obst2}, let $m \in \{1,\ldots,n-1\}$, and let $\Sigma_m:=\cup_{t>0}\Sigma_{m,t}$. Then $\Sigma_{m}=\Sigma_m^g\cup\Sigma_m^a$,
where:\\
- $\Sigma_m^g$ can be locally covered by a $C^{1,1}$
$m$-dimensional manifold;\\
- $\Sigma^a_m$ is a relatively open subset of $\Sigma_{m}$ satisfying ${\rm dim}_{\mathcal H}(\Sigma^a_m)\leq m-1$.
\end{theorem}

The extension from the elliptic to the parabolic problem is far from trivial, as it requires relating the behavior of the solution at different times for different singular points. This involves both finer analytic arguments and a series of new covering-type theorems that allow us to take care of sets coming from different times.
This result, besides considerably improving the previous knowledge on the structure of the free boundary for the Stefan problem, can be used to estimate the size of the set of singular times. 

Let $\Sigma_t:=\cup_{m=0}^{n-1}\Sigma_{m,t}$, and define the set of 
 {\it singular times}
$$
\mathcal S:=\{t>0\,:\,\Sigma_{t}\neq \emptyset\}.
$$
Because singular points are a closed subset of the free boundary,  if we ensure that the free boundary is contained in a bounded domain then the set $\Sigma$ is compact, in which case $\mathcal S$ is a compact subset of $\R^+$. In particular, if $t_0 \not \in \mathcal S$ then there exists $\tau_0>0$ such that
$$\partial\{u>0\}\cap \bigl((t_0-\tau_0,t_0+\tau_0)\times \R^n\bigr) \quad \text{is an analytic hypersurface}.
$$
A fundamental question is to estimate the size of $\mathcal S$.

Recalling that $\partial \{u>0\}$ coincides with the free boundary for the Stefan problem (recall Section \ref{sec:stefan obst}),
in \cite{FRS} we prove the following result:
\begin{theorem}
\label{thm:stefan}
Let $\Omega\subset \R^3$ be a bounded domain, and let $\theta$ solve the Stefan problem in $\mathbb R^+\times \Omega$, with $\theta(t)>0$ on $\partial \Omega$. Then
$$
\partial\{\theta>0\}\cap \bigl((\mathbb R^+\setminus \mathcal S)\times  \Omega\bigr)\quad \text{is analytic,}
$$ 
where $\mathcal S\subset \mathbb R^+$ is a compact set satisfying
$$
{\rm dim}_{\mathcal H}(\mathcal S)\leq \frac12.
$$
\end{theorem}
At least to our knowledge, this is the first result on the size of the singular times.

\subsection{Generic obstacle problems}
A famous conjecture by Schaeffer on the elliptic obstacle problem states that, generically, the set of singular points in the free boundary should be empty \cite{Sch76}. This result has been proved in dimension 2 by Monneau \cite{M03}.

To attack this problem, given a domain $\Omega\subset \mathbb R^n$, let $t\mapsto f_t \in C^0(\partial\Omega)$ be a one parameter family of nonnegative functions such that
$\partial_t f_t >0$ inside the set $\{f_t>0\}$.
For instance one may consider $f_t(x):=f(x)+t$ as in \cite{M03}, but many other choices are possible.

Then, for any $t$ we can consider $u_t$ the solution of the obstacle problem
$$
\left\{
\begin{array}{ll}\Delta u_t=\chi_{\{u_t>0\}}&\text{in }\Omega,\\
u_t \geq 0&\text{in }\Omega,\\
u_t=f_t &\text{on }\partial\Omega.
\end{array}
\right.
$$
It is interesting to observe that this one-parameter family of elliptic obstacle problems can be used to investigate the regularity of the free boundary both in the study of injection of fluid into a finite
Hele-Shaw cell \cite{Ell81}
and in the two dimensional annular electrochemical maching problem \cite{Ell80}.

Note that, since the boundary data are increasing, 
$$
u_t\geq u_s\qquad \text{for $t \geq s$.}
$$
For each $t$ we define
$\Sigma_{m,t}$ as the $m$-th stratum of the singular points for $u_t$.
Of course Theorem \ref{thm:main} applies to each $u_t$. However, exploiting the monotonicity with respect to $t$, as in the parabolic case we can prove that the very same theorem holds with $\Sigma_{m,t}$ replaced by $\cup_{t>0}\Sigma_{m,t}$ (actually, we can prove even a finer version of that result).

To obtain this, several new difficulties arise with respect to the parabolic case. Indeed, while on the one hand the fact that each $u_t$ solves an elliptic problem simplifies the analysis, on the other hand much more work is needed to relate the behavior of different solutions $u_t$ and $u_s$ are different singular points.

As an application, 
define $\Sigma_t:=\cup_{m=0}^{n-1}\Sigma_{m,t}$.
Then we can prove the following result:
\begin{theorem}
\label{thm:Schaeffer}
Let $u_t$ be as before. Then, for a.e. $t$, the singular set $\Sigma_t$ satisfies
$$
{\rm dim}_{\mathcal H}(\Sigma_t)\leq n-4.
$$
In particular, for $n \leq 3$ we have $\Sigma_t=\emptyset$ for a.e. $t$. 
\end{theorem}

Notice that, by the discussion above, this result implies the validity of Schaeffer's conjecture in dimension $n\leq 3$.

Actually, as in the case 
of Theorem \ref{thm:stefan},
for $n=2,3$ we can give an estimate on the Hausdorff dimension of the set of  singular times: if we define 
$$
\mathcal S:=\{t>0\,:\,\Sigma_t\neq\emptyset\},
$$
then
$$
\text{${\rm dim}_{\mathcal H}(\mathcal S)\leq 1/4\,$ for $n=2$},\qquad 
\text{${\rm dim}_{\mathcal H}(\mathcal S)\leq 1/2\,$ for $n=3$.}
$$
Recalling the connection for $n=2$ to the Hele-Shaw flow \cite{Ell81} and the electrochemical maching problem \cite{Ell80} discussed above, we deduce that in these problems the free boundary is smooth outside a closed set of singular times of dimension at most 1/4.

\end{document}